\documentclass[a4paper,11pt]{amsart}

\usepackage{hyperref}

\usepackage{amsmath,amsthm,amssymb,latexsym,amsfonts,wrapfig,yfonts,mathrsfs,mathtools}
\usepackage[latin1]{inputenc}
\usepackage[english,activeacute]{babel}
\usepackage{graphicx,xcolor,esint}
\usepackage[mathscr]{euscript}
\usepackage{geometry,upgreek,bbm} 
\usepackage[normalem]{ulem}
\usepackage{enumerate}
 \usepackage{mathrsfs}
 \usepackage{tikz}
 \usepackage{setspace}

\usepackage{parskip}
\usepackage{hyperref}
\usepackage{cleveref}

\newcommand{\norm}[1]{\left\lVert#1\right\rVert}

\def \N{\mathbb N}
\def \Z{\mathbb Z}

\def \R{\mathbb R}
\def \C{\mathbb C}

\def \P{{\mathbb P}}
\def \pa{{\partial}}

\def \O{\mathcal{O}}
\def \T{\mathbb{T}}

\def \E{\mathbb{E}}

\newcommand{\Ac}{\mathcal A}
\newcommand{\Bc}{\mathcal B}
\newcommand{\Cc}{\mathcal C}

\newcommand{\Gc}{\mathcal{G}}
\newcommand{\Hc}{\mathcal H}
\newcommand{\Ic}{\mathcal I}
\newcommand {\Jc}{\mathcal{J}}

\renewcommand{\Mc}{\mathcal M}

\newcommand{\Rc}{\mathcal R}

\renewcommand{\bar}{\overline}

\numberwithin{equation}{section}

\theoremstyle{plain}
\newtheorem{thm}{Theorem}[section]

\newtheorem{lem}[thm]{Lemma}
\newtheorem{prop}[thm]{Proposition}
\newtheorem{cor}[thm]{Corollary}
\newtheorem{rk}[thm]{Remark}
\newtheorem{claim}[thm]{Claim}
\theoremstyle{definition}

\theoremstyle{remark}

\theoremstyle{plain}

\theoremstyle{remark}

\title[Resonant LDP for beating NLS]{Resonant large deviations principle for the beating NLS equation}

\author[R. Grande]{Ricardo Grande}
\address{International School for Advanced Studies (SISSA), Via Bonomea 265, 34136, Trieste, Italy}
\email{rgrandei@sissa.it}

\begin{document}
	
\begin{abstract}
We prove a large deviations principle for the solution to the beating NLS equation on the torus with random initial data supported on two Fourier modes. When these modes have different initial variance, we prove that the resonant energy exchange between them increases the likelihood of extreme wave formation. Our results show that nonlinear focusing mechanisms can lead to tail fattening of the probability measure of the sup-norm of the solution to a nonlinear dispersive equation.
\end{abstract}

\maketitle

\section{Introduction}

Extreme waves are often defined to be isolated water waves whose amplitude exceeds twice the characteristic wave height expected for the given surface conditions \cite{DGV,Dysthe}. These waves are very difficult to predict, due in no small part to the many possible mechanisms of formation \cite{Onorato-mechanism}, which are the subject of much debate and strongly depend on the physical context (e.g. the depth of the water body). Two of the main mechanisms are believed to be linear superposition and nonlinear focusing. Linear superposition is a constructive interference phenomenon: the phases of many waves coincide at a certain point of space, giving rise to a large peak at that point. Conversely, the focusing mechanism is a nonlinear mechanism which consists of a substantial exchange of energy between waves of different wavenumbers resulting in a large increase in the amplitude of one of them.

In a recent series of works, Dematteis, Grafke, Onorato and Vanden-Eijnden proposed a novel statistical approach to the study of extreme waves \cite{OnoratoDGV,DGV,DGV2}. They considered two popular models\footnote{Such models are asymptotic approximations of the water waves system in the modulational regime and in infinite depth. While the well-posedness theory of the cubic NLS equation is well-understood, that of the Dysthe equation in $\T$ seems to be an open problem, in contrast to the $\R^2$ case which is better understood, cf.\ \cite{GKS,Saut}.}
for wave propagation and interaction on a bounded periodic domain: the Dysthe equation and the cubic NLS equation with random initial data. Their simulations suggest the existence of a typical profile for extreme waves, which can be formally and numerically predicted by using techniques from Large Deviations theory. This exciting combination of ideas uncovers a path towards identifying the most likely mechanism of formation of extreme waves, as well as the most likely profile of such waves. 

From a mathematical viewpoint, this problem is connected to the evolution of probability measures under nonlinear dynamics. The two main mechanisms of formation of extreme waves are believed to have different implications from a statistical viewpoint. Starting from sufficiently smooth Gaussian initial data, the linear superposition mechanism is not expected to substantially vary the statistics of the system. However, if the system is sufficiently energetic, the nonlinear focusing mechanism is conjectured (and empirically observed) to increase the likelihood of extreme waves, resulting in non-Gaussian fat-tailed statistics \cite{OnoratoDGV}. More precisely, it is believed that the initial Gaussian measure converges to a non-Gaussian invariant measure of the system, which features fat tails. 

While a rigorous proof of such strong conjectures is, for the moment, out of reach, the goal of this paper is to prove that focusing mechanisms can indeed lead to tail fattening of the probability measure in the case of a toy model. In particular, we will show the existence of an ``equilibrium'' for the tails of the probability measure, featuring fatter tails than the initial measure, which is reached after a transient regime where the tail probability grows away from the linear statistics. This description is thus consistent with the conjectured convergence to a supposed fat-tailed, invariant measure of the system which, in the case of our toy model, remains sub-Gaussian.

\subsection{Statement of results}

We consider the following NLS-type equation as a toy model for energy exchange between Fourier modes:
\begin{equation}\label{eq:beating}
\begin{cases}
i \pa_t u + \pa_x^2 u = 2\, \cos (2x)\, |u|^2 u, \quad x\in\T,\\
u(t,x)|_{t=0}= \varepsilon\,  ( \alpha\, e^{ix} + \beta\, e^{-ix}).
\end{cases}
\end{equation}
Here, $\alpha$ and $\beta$ are complex, independent Gaussian random variables with zero mean and variance
\begin{equation}
\sigma_{\alpha}^2=\E |\alpha|^2, \hspace{1cm} \sigma_{\beta}^2=\E |\beta|^2.
\end{equation}

Note that the mass and the energy are conserved quantities:
\begin{equation}\label{eq:conserved}
\Mc  = \int_{\T} |u(t,x)|^2 \, dx, \qquad 
\Hc   = \int_{\T} |\pa_x u(t,x)|^2 + 2\, \int_{\T} \cos (2x)\, |u|^4\, dx.
\end{equation}

One may use the conservation of the $L^2$-norm, together with a classical local well-posedness result by Bourgain in $X^{s,b}$ spaces \cite{Bourgain,Bourgainbook}, in order to show that \eqref{eq:beating} is globally well-posed for initial data in $L^2(\T)$. In particular, the sign of the nonlinearity plays no role, and we have thus chosen a positive sign. 

This equation, known as the beating NLS equation, was originally introduced by Gr\'ebert and Villegas-Blas \cite{beating}, who proved a beating effect in the deterministic setting: the dynamics are concentrated in the Fourier modes $u_1$ and $u_{-1}$, which exchange energy periodically. 

The main theorem of this paper shows that this resonant exchange of energy between Fourier modes results in tail fattening of the probability distribution of the solution.

\begin{thm}\label{thm:main_intro}
Fix $z_0>0$,  $\delta\in (0,1)$ and
\begin{equation}\label{eq:gamma_intro}
0\leq \gamma < \frac{5}{2}(1-\delta).
\end{equation}
Let $u_{\varepsilon}$ be the solution to \eqref{eq:beating} where $\alpha$ and $\beta$ are complex, independent Gaussian random variables with zero mean and variance $\E |\alpha|^2 = \sigma_{\alpha}^2 >  \sigma_{\beta}^2 = \E |\beta|^2$. Then for times $t=c\, \varepsilon^{-\gamma}$, $c>0$, the following large deviations principle holds:
\begin{equation}\label{eq:main_intro}
\lim_{\varepsilon\rightarrow 0} \ \varepsilon^{2\delta} \log \P \left( \sup_{x\in\T} |u_{\varepsilon} (t,x)| \geq z_0 \varepsilon^{1-\delta}\right) = \begin{cases}
-\dfrac{z_0^2}{\sigma_{\alpha}^2+\sigma_{\beta}^2}  &  \mbox{if}\quad \gamma< 2(1-\delta),\\
\vspace{-0.3cm}\\
-\dfrac{z_0^2}{2\sigma_{\alpha}^2} & \mbox{if}\quad 2(1-\delta)<\gamma< \frac{5}{2}(1-\delta).
\end{cases}
\end{equation}
\end{thm}

\begin{rk}\label{rk:transient} Note that the transient regime $t=\tau\, \varepsilon^{-2(1-\delta)}$, $\tau\in (0,\infty)$, is not covered by \Cref{thm:main_intro}. In \Cref{thm:resonantLDP_diff_variance}, we prove that:
\begin{equation}\label{eq:transient_intro}
\begin{split}
\limsup_{\varepsilon\rightarrow 0} \ \varepsilon^{2\delta} \log \P \left( \sup_{x\in\T} |u_{\varepsilon}(t,x)| \geq z_0 \varepsilon^{1-\delta}\right) & \leq -\frac{z_0^2}{2\sigma_{\alpha}^2}\\
\liminf_{\varepsilon\rightarrow 0} \ \varepsilon^{2\delta} \log \P \left( \sup_{x\in\T} |u_{\varepsilon}(t,x)| \geq z_0 \varepsilon^{1-\delta}\right) & \geq -\frac{\Jc (\tau,z_0)^2}{\sigma_{\alpha}^2},
\end{split}
\end{equation}
where $\Jc$ is a c\`adl\`ag function which tends to $z_0/\sqrt{2}$ as $\tau\rightarrow\infty$, cf.\ \eqref{eq:implicit_J}-\eqref{eq:asymp_jumps}. In general, the upper and lower bounds in \eqref{eq:transient_intro} do not coincide, except as $\tau\rightarrow\infty$, which is a key ingredient of the proof of \Cref{thm:main_intro}. However, we believe that the techniques presented in Sections \ref{sec:implicit} and \ref{sec:Laplace} could be adapted to produce a LDP in this transient regime.
\end{rk}

\begin{rk} When the variance of the initial Fourier modes is equal, $\sigma_{\alpha}^2 =\sigma_{\beta}^2=\sigma^2$, the rate functions on the right-hand side of \eqref{eq:main_intro} coincide, suggesting that one may not see the effect of energy exchanges from a statistical viewpoint. We give a rigorous proof of this edge case in \Cref{thm:main_equal_var}, which extends \Cref{thm:main_intro} to the equal variance case:
\begin{equation}\label{eq:main_intro_samevar}
\lim_{\varepsilon\rightarrow 0} \ \varepsilon^{2\delta} \log \P \left( \sup_{x\in\T} |u_{\varepsilon}(t,x)| \geq z_0 \varepsilon^{1-\delta}\right) =-\frac{z_0^2}{2\sigma^2}  \qquad \mbox{for}\ t=c\, \varepsilon^{-\gamma},\ 0\leq \gamma< \frac{5}{2}(1-\delta).
\end{equation}
\end{rk}

\begin{rk} \Cref{thm:main_intro} allows $t=t(\varepsilon)$ to depend on $\varepsilon$ in more complex ways beyond a polynomial $c\varepsilon^{-\gamma}$. In particular, one can replace the condition $\gamma< 2(1-\delta)$ (resp. $>2(1-\delta)$) by $\lim_{\varepsilon\rightarrow 0^{+}} t(\varepsilon)\, \varepsilon^{2(1-\delta)} = 0$ (resp. $=\infty$). A precise statement of these sharp results can be found in \Cref{thm:LDP_diff_linear} and \Cref{thm:super_resonantLDP_diff_variance}.
\end{rk}

\begin{rk}\label{rk:longer_gamma} The restriction \eqref{eq:gamma_intro} stems from a normal form result which allows us to approximate the effective dynamics of \eqref{eq:beating} by a simpler Hamiltonian system under the restriction \eqref{eq:gamma_intro}, cf.\ \Cref{sec:normal_form}. This restriction may be removed if one instead considers the evolution of the $2$-dimensional Hamiltonian system given by
\begin{equation}\label{eq:2DHam_intro}
\begin{split}
\begin{cases}
\pa_t u_1 = -i\, \dfrac{\pa \Gc}{\pa \bar{u}_1}, \qquad 
\pa_t u_{-1}  = -i\, \dfrac{\pa \Gc}{\pa \bar{u}_{-1}},\\
u_1(0) = \varepsilon\, \alpha, \hspace{1.35cm} u_{-1}(0)  =  \varepsilon\, \beta.
\end{cases}
\end{split} 
\end{equation}
with 
\[
\Gc = |u_1|^2 + |u_{-1}|^2 + 2(u_1 \overline{u_{-1}} + u_{-1} \overline{u_1})\, (|u_1|^2 + |u_{-1}|^2).
\]
Our techniques allow us to prove the LDP \eqref{eq:main_intro} for the solution $u(t,x)=u_1(t)e^{ix} + u_{-1}(t) e^{-ix}$ to \eqref{eq:2DHam_intro} for $t=c\varepsilon^{-\gamma}$, for any $c>0$ and any $0\leq \gamma <\infty$.
\end{rk}

\subsection{Some background literature}

As previously mentioned, the statistical study of extreme waves is connected to the evolution of probability measures under nonlinear dynamics. 
While much attention has been devoted to the existence and construction of invariant measures for dispersive PDEs, the problem of the out-of-equilibrium evolution of measures under nonlinear, dispersive dynamics is comparatively less explored.

Considerable progress on this topic has recently been achieved in the context of wave turbulence theory, where one is interested in the evolution of the second order moment of the solution to dispersive equations, see for instance \cite{DengHani2,DengHani} and \cite{GrandeHani,HRST} in the random data and stochastic settings, respectively.
On the other hand, there have been recent advances on certain qualitative properties of the evolution of probability measures under Hamiltonian PDEs, such as their absolute continuity with respect to the initial Gaussian measure \cite{Tzvetkov}. These techniques can yield quantitative bounds on the Radon-Nikodym derivative, such as integrability in various $L^p$ spaces, which can be used to give upper bounds on the tails of the probability measure. Unfortunately, such bounds are not (yet) sharp enough to obtain large deviations principles (LDP), which is the subject of this manuscript.

In a previous work with M. Garrido, K. Kurianski and G. Staffilani \cite{GGKS}, we proved a large deviations principle for the cubic NLS equation on $\T$ in the weakly nonlinear setting, which features a constant in time rate function. More precisely, the fact that the moduli of the Fourier coefficients are almost invariant over long timescales, which is connected to the complete integrability of the system, implies that nonlinear interactions are mostly restricted to the phases of the Fourier coefficients. It follows from these results that, from a statistical viewpoint, linear superposition is the main driving mechanism of extreme waves in this context.

The current manuscript deals with the opposite setting from \cite{GGKS}, namely an equation where the Fourier coefficients of the solution exchange large amounts of energy. These exchanges of energy can create growth of the moduli of the Fourier coefficients, which becomes the main driving mechanism of extreme wave formation. To the best of our knowledge, this is the first rigorous mathematical result where the focusing mechanism is proved to increase the likelihood of appearance of extreme waves. It is interesting to mention that the tails in our result grow while remaining sub-Gaussian (i.e. the right-hand side of \eqref{eq:main_intro} depends on $z_0$ quadratically). Proving convergence to non-Gaussian fat-tailed probability measure starting from an initial Gaussian measure under \emph{any} nonlinear dispersive PDE remains a very interesting open question.

\subsection{Strategy of proof}

The results presented in \Cref{thm:main_intro} are based on three main ideas: 

\subsubsection{Approximation of the dynamics}

Putting the Hamiltonian \eqref{eq:conserved} in normal form up to order $4$ yields an explicit approximation to the solution $u_{\varepsilon}$ to \eqref{eq:beating} in terms of the random variables in the initial data. In particular, this yields an explicit formula for the $L^{\infty}$-norm of the solution in terms of the initial random variables:
\begin{equation}\label{eq:sup_intro}
\begin{split}
\sup_{x\in\T} |u_{\varepsilon}(t,x)| =&\ \varepsilon\, \sqrt{|\alpha|^2 \cos^2 \left( 2\varepsilon^2 t\, (|\alpha|^2+ |\beta|^2)\right) + |\beta|^2  \sin^2 \left( 2\varepsilon^2 t\, (|\alpha|^2+ |\beta|^2)\right) } \\
& + \varepsilon\, \sqrt{|\beta|^2 \cos^2 \left( 2\varepsilon^2 t\, (|\alpha|^2+ |\beta|^2)\right) + |\alpha|^2  \sin^2 \left( 2\varepsilon^2 t\, (|\alpha|^2+ |\beta|^2)\right)} + \O( \varepsilon^{\frac{3}{2}(1-\delta)})
\end{split}
\end{equation}
which is valid on a set of large probability\footnote{Whose complement is negligible with respect to the large deviations principle in \Cref{thm:main_intro}.}, up to timescales $t=c\varepsilon^{-\gamma}$ for any $c>0$ and for $\gamma$ in \eqref{eq:gamma_intro}.

\subsubsection{Sharp bounds for Gaussian integrals}

After a simple rescaling, one may use \eqref{eq:sup_intro} to obtain an approximation to the tail-probability of the $L^{\infty}$-norm of the solution to \eqref{eq:beating},
\begin{equation}\label{eq:Gaussint_intro}
\P \left( \sup_{x\in\T} |u_{\varepsilon} (t,x)| \geq z_0 \varepsilon^{1-\delta}\right)  \approx  \int_{\Ac (\tau,z_0)}  \frac{4\varepsilon^{-4\delta}\, a\, b}{\sigma_{\alpha}^2\, \sigma_{\beta}^2}  \exp \left( - \varepsilon^{-2\delta}\, \frac{a^2}{\sigma_{\alpha}^2}  - \varepsilon^{-2\delta}\, \frac{b^2}{\sigma_{\beta}^2} \right) \, da \, db
\end{equation}
where $\tau=t \varepsilon^{2(1-\delta)}$ and
\begin{multline}\label{eq:A_set_intro}
\Ac(\tau,z_0) = \left\lbrace (a,b)\in [0, c]^2 \ \Big | \   \sqrt{a^2 \cos^2 (2\tau \,[a^2 + b^2]) + b^2  \sin^2 (2 \tau \,[a^2 +b^2]) } \right. \\
\left. + \sqrt{b^2 \cos^2 (2\varepsilon^2 \tau \,[a^2 + b^2]) + a^2 \sin^2 (2 \tau \,[a^2 + b^2]) } \geq z_0 \right\rbrace ,
\end{multline}
where $c=c(z_0)>10z_0$ is sufficiently large, see \Cref{sec:Laplace} for details.

We note that when $\tau \ll 1$, the set $\Ac(\tau,z_0)$ can be approximated by $\{(a,b)\mid |a|+|b|\geq z_0\}$ which allows an explicit integration of the right-hand side of \eqref{eq:Gaussint_intro}. Note also that $\Ac(\tau,z_0)\subset \{ (a,b)\mid \sqrt{2}\, \sqrt{|a|^2 + |b|^2} \geq z_0\}$, which leads to a simple upper bound of \eqref{eq:Gaussint_intro} that happens to be sharp as $\tau\rightarrow\infty$. 

Obtaining sharp lower bounds for \eqref{eq:Gaussint_intro} constitutes the main novelty in this work. Despite the relative simplicity of the integrand, the time-dependent domain $\Ac(\tau,z_0)$ is quite intricate and generally non-convex, and is thus not amenable to a classical Laplace-type approximation. To overcome these difficulties, we show that $\Ac(\tau,z_0)$ approximately contains a set of the form 
\begin{equation}\label{eq:intregion_intro}
\{ a\mid F(\tau,a,z_0)\geq 0\}\times [0,\varepsilon]_b, 
\end{equation}
where
\begin{equation}\label{eq:implicit_intro}
F(\tau, a, z_0)= a\, \left( |\cos\left(2\tau\, a^2\right)| + |\sin\left(2\tau \,a^2\right)|\right) - z_0.
\end{equation}

Classical Laplace approximation arguments suggest that the integral \eqref{eq:Gaussint_intro} should be dominated by the minimal solution $a$ to the equation $F(\tau,a,z_0)=0$.

\subsubsection{Analysis of the implicit equation}

As a consequence of the analysis mentioned above, the zeroes of the function $F$ defined in \eqref{eq:implicit_intro} are of fundamental importance. In \Cref{sec:implicit}, we show that, for each fixed $(\tau,z_0)$, there are a large but finite number of solutions $a (\tau,z_0)$ to the implicit equation $F(\tau,a,z_0)=0$, which can be ordered from smallest to largest. For fixed $(\tau,z_0)$ it is thus possible to define the minimal solution $Y (\tau,z_0)$.

Unfortunately, each solution $a(\tau,z_0)$ ceases to exist due to $\pa_{\tau} a$ blowing up at a certain time $\tau^{\infty}(z_0)$, which admits an explicit asymptotic development, c.f.\ \eqref{eq:asymp_jumps}. These ``blow-up'' times, which are roughly in arithmetic sequence, constitute points in which the minimal solution collides with the second smallest solution. They are therefore jump points of the minimal solution $Y (\tau,z_0)$. 

At points of continuity of $Y (\tau,z_0)$, one may exploit these results to give a lower bound to \eqref{eq:Gaussint_intro} of the form 
\begin{equation}\label{eq:Gaussint_approx_intro}
\P \left( \sup_{x\in\T} |u_{\varepsilon} (t,x)| \geq z_0 \varepsilon^{1-\delta}\right) \geq C(\tau,z_0)\,  \varepsilon^2\, \exp \left( - \varepsilon^{-2\delta}\, \frac{Y(\tau,z_0)^2}{\sigma_{\alpha}^2}\right).
\end{equation}
The constant $C(\tau,z_0)$ depends on the distance from the minimal solution $Y(\tau,z_0)$ to the next smallest solution to $F(\tau,a,z_0)=0$. In particular, this constant vanishes at the jump points $\tau^{\infty}(z_0)$ of $Y(\tau,z_0)$. This small obstacle can be overcome by exploiting the third and fourth smallest solutions to $F(\tau,a,z_0)=0$, instead of the two smallest ones. In particular, this approach gives rise to the results presented in \Cref{rk:transient} which cover the transient regime $t=\tau\varepsilon^{-2(1-\delta)}$ with $\tau\in (0,\infty)$.

Unfortunately, these arguments are insufficient to handle the regime $t\gg \varepsilon^{-2(1-\delta)}$, as the distance between any two consecutive solutions to $F(\tau,a,z_0)=0$ shrinks to zero as $\tau\rightarrow\infty$.
We overcome this difficulty by simultaneously considering the four smallest solutions, $a_1,\ldots ,a_4$, to $F(\tau,a,z_0)=0$, exploiting the following inclusion in the region of integration \eqref{eq:intregion_intro}
\[
\left( [a_1(\tau,z_0), a_2(\tau,z_0)]\cup [a_3(\tau,z_0), a_4(\tau,z_0)] \right) \times [0,\varepsilon]_b \subset \{ a\mid F(\tau,a,z_0)\geq 0\}\times [0,\varepsilon]_b.
\]
More precisely, we show that, as $\tau\rightarrow\infty$, both $a_1$ and $a_3$ tend\footnote{Strictly speaking, these solutions do not exist globally in $\tau$. However, we show that they are close to $z_0/\sqrt{2}$ for all relevant times of existence.}
 to $z_0/\sqrt{2}$. Moreover, the sum of the distances between these solutions admits a uniform lower bound of the form:
\[
|a_2(\tau,z_0)-a_1(\tau,z_0)| + |a_4(\tau,z_0) - a_3(\tau,z_0)| \gtrsim_{z_0} \tau^{-2}\gtrsim \varepsilon,
\]
see \Cref{thm:asymp_increment} and \Cref{thm:Y_inf_limit} for the precise statements of these results. 

These ideas allow us to replace $C(\tau,z_0)$ by $\varepsilon$ and $Y(\tau,z_0)$ by $z_0/\sqrt{2}$ in \eqref{eq:Gaussint_approx_intro}, which ultimately leads to \Cref{thm:main_intro} in the regime $t\gg \varepsilon^{-2(1-\delta)}$.

\subsection{Outline} In \Cref{sec:normal_form}, we present some background on the beating NLS equation and its normal form, originally introduced in \cite{beating}, which allow us to approximate the nonlinear dynamics. In \Cref{sec:LDP}, we prove the main upper bounds for the tail probability in \Cref{thm:main_intro}, as well as sharp lower bounds for sub-resonant timescales. In \Cref{sec:implicit}, we conduct an in-depth study of the implicit equation \eqref{eq:implicit_intro}, including finding its minimal solution, deriving asymptotic developments in time, and giving lower bounds on the distance between consecutive solutions. Finally, in \Cref{sec:Laplace}, we prove several lower bounds on Gaussian integrals which include the proof of \Cref{thm:main_intro} for super-resonant timescales.

\subsection{Notation} 

We write $a\lesssim b$ to indicate an estimate of the form $a\leq c\, b$ for some constant $c>0$ which may change from line to line. The inequality $a\lesssim_d b$ indicates that the implicit constant $c$ depends on the parameter $d$. 
 For functions $f,g:[0,\infty)\rightarrow \R$ of a variable $\varepsilon>0$, we write $f(\varepsilon)\ll g(\varepsilon)$ whenever $\lim_{\varepsilon\rightarrow 0^{+}} f(\varepsilon)/g(\varepsilon) = 0$. We also employ the big $\O$ notation, i.e. $f(\varepsilon)=\O(g(\varepsilon))$ means that $f(\varepsilon)\lesssim g(\varepsilon)$. Similarly, $f(\varepsilon)= o(g(\varepsilon))$ means that $f(\varepsilon) \ll g(\varepsilon)$.

We denote by $\R_{+}=[0,\infty)$. For periodic functions $f: \T=[0,2\pi]\rightarrow \C$, we define $L^p (\T)$ , $1\leq p\leq \infty$, to be the space of Lebesgue-integrable functions such that 
\[
\norm{f}_{L^p}= \left( \int_0^{2\pi} |f(x)|^p \, dx \right)^{1/p} <\infty,
\]
with the standard modifications when $p=\infty$. We will often write functions in $f\in L^2(\T)$ in Fourier series
\[
f(x) = \sum_{k\in\Z} f_k \, e^{ikx}, \qquad f_k = \frac{1}{2\pi}\, \int_0^{2\pi} f(x) e^{-ikx}\, dx.
\]

Given a sequence of complex numbers $c = (c_k)_{k\in\Z}$, we say that $c\in \ell^p(\Z)$ if 
\[
\norm{c}_{\ell^p} = \left( \sum_{k\in\Z} |c_k|^p \right)^{1/p} <\infty.
\]

\subsection*{Acknowledgements}

The author would like to thank A. Maspero and N. Visciglia for some useful conversations.

\section{Normal form and approximate dynamics}\label{sec:normal_form}

Equation \eqref{eq:beating} is a Hamiltonian system with Hamiltonian
\begin{equation}
\Hc = \Hc_2 + \Hc_4= \sum_{k\in\Z} k^2 u_k v_k  + \sum_{k_1+k_2-k_3-k_4 = \pm 2} u_{k_1} u_{k_2} v_{k_3} v_{k_4}
\end{equation}
where $v_k = \overline{u_k}$, in such a way that the dynamics are given by $\pa_t u_k = \{ u_k, \Hc\}$ with the Poisson bracket:
\[
\{ F, G\} =-i\, \sum_{k\in\Z} \left( \frac{\pa F}{\pa u_k} \, \frac{\pa G}{\pa v_k} - \frac{\pa G}{\pa u_k} \, \frac{\pa F}{\pa v_k}\right).
\]

In particular, we may write:
\[
\pa_t \begin{pmatrix}
u_k \\ v_k 
\end{pmatrix}= -i\, X_{\Hc}(u,v)_k  := -i\, J\cdot \begin{pmatrix} \dfrac{\pa \Hc}{\pa u_k} \\ \vspace{-0.25cm} \\ \dfrac{\pa \Hc}{\pa v_k} \end{pmatrix}, 
\qquad J=\begin{pmatrix}
0 & 1 \\
-1 & 0
\end{pmatrix}.
\]

Since we only have two dominant modes, we may fix $\rho\geq 0$ and work in the following subspace of $\ell^1(\Z)$
\begin{equation}
Y_{\rho}  = \left\lbrace u : \T \rightarrow \C\  \Big | \ u(x) = \sum_{k\in\Z} u_k e^{ikx}, \quad \norm{u}_{Y_{\rho}} := \sum_{k\in\Z} e^{\rho |k|} |u_k| <\infty \right\rbrace .
\end{equation}
Note that this is a Banach space which enjoys the algebra property.

In order to study the dynamics of equation \eqref{eq:beating}, we exploit the following result by Gr\'ebert and Villegas-Blas \cite{beating}, which follows from the standard Birkhoff normal form procedure.

\begin{thm}\label{thm:normal_form}
Fix $\rho\geq 0$. There exist some $\varepsilon_0 (\rho)>0$ and a canonical change of variable $\Phi$ such that for all $\varepsilon\leq \varepsilon_0$, $\Phi$ maps the ball $B_{Y_{\rho}}(\varepsilon)$ into $B_{Y_{\rho}}(2\varepsilon)$. Moreover,
 \[
 \Hc \circ \Phi =\widehat{\Hc}= \Hc_2 +\widehat{\Hc}_{4,1} + \widehat{\Hc}_{4,2} + \widehat{\Hc}_{4,3} + \Rc_5
 \]
where
\begin{enumerate}[(i)]
\item $\Hc_2 = \sum_{k\in\Z} k^2 u_k v_k $.
\item $\widehat{\Hc}_4 = \widehat{\Hc}_{4,1} +\widehat{\Hc}_{4,2} + \widehat{\Hc}_{4,3}$ is a polynomial of order 4 satisfying $\{ \Hc_2 ,\widehat{\Hc}_4\}=0$.
\item $\widehat{\Hc}_{4,1} = 2(u_1 v_{-1} + u_{-1} v_1)\, (2 \sum_{k\neq \pm 1} u_k v_k + u_1v_1 + u_{-1}v_{-1})$.
\item $\widehat{\Hc}_{4,2} = 4 (u_2 u_{-1} v_{-2} v_1 + u_{-2} u_1 v_2 v_{-1})$.
\item $\widehat{\Hc}_{4,3}$ contains the terms in $\widehat{\Hc}_4$ involving at most one mode of index $1$ or $-1$.
\item $\Rc_5$ is the remainder of order 5 satisfying $\norm{X_{\Rc_5}(u)}_{Y_{\rho}} \lesssim \norm{u}_{Y_{\rho}}^4$.
\item $\norm{\Phi (u)-u}_{Y_{\rho}} \lesssim_{\rho} \norm{u}^2_{Y_{\rho}}$.
\end{enumerate}
\end{thm}

This theorem will allow us to study the dynamics of the Hamiltonian system given by $\widehat{\Hc}$, i.e.
\begin{equation}\label{eq:phi_beating}
\begin{cases}
\pa_t u_k = \{ u_k , \widehat{\Hc}\} \qquad \forall k \in\Z,\\
|u_1(0)|, |u_{-1}(0)|< c\, \varepsilon^{1-\delta}, \quad u_k(0)=0 \quad k\in\Z-\{\pm 1\},
\end{cases}
\end{equation}
where $\delta\in (0,1)$ and $c>0$ will later be fixed and $\varepsilon>0$ is sufficiently small so that \Cref{thm:normal_form} applies.

Adapting the results in \cite{beating} to our setting, we will show that most of the dynamics of \eqref{eq:phi_beating} are restricted to the Fourier modes $\pm 1$ for long times. In order to prove this, we introduce the following quantities for $n\in\N$
\begin{equation}\label{eq:quantities}
\begin{split}
J_n & = u_n v_n+ u_{-n} v_{-n},\\
K_n& = u_n v_{-n} + u_{-n} v_n,\\
L_n & =  u_n v_n - u_{-n} v_{-n}.
\end{split}
\end{equation}
 Let us also denote by $M$ the mass $\sum_k u_k v_k$. 

The first consequence of \Cref{thm:normal_form} is the almost conservation of $J_n$ over long timescales.

\begin{cor}\label{thm:conserved_J} Fix a small $\kappa>0$, $\delta\in (0,1)$ and suppose that $J_1(0)\leq c\,\varepsilon^{2(1-\delta)}$ for some $c>0$. Then there exists some $\varepsilon_0 = \varepsilon_0 (c,\delta)$ such that for all $\varepsilon<\varepsilon_0$ the following holds. If $t\lesssim \varepsilon^{-\frac{5}{2}(1-\delta)+\kappa}$, then we have that the solution to \eqref{eq:phi_beating} satisfies
\begin{equation}\label{eq:no_energy}
\begin{split}
J_1 (t) & = J_1 (0) + \O (\varepsilon^{\frac{5}{2}(1-\delta)+\kappa}),\\
\sum_{n=2}^{\infty} J_n (t) & \lesssim \varepsilon^{3(1-\delta)} .
\end{split}
\end{equation}
\end{cor}
\begin{proof}
For all $n\in\N$, one can verify that
\[
\{ J_n , \Hc_2 \} = \{ J_n , \widehat{\Hc}_{4,1} \} = \{ J_n , \widehat{\Hc}_{4,2}\} =0.
\]
By \Cref{thm:normal_form}, 
\[
\pa_t  J_n = \{ J_n ,  \widehat{\Hc}_{4,3} \} + \{ J_n , \Rc_5\}.
\]
We prove \eqref{eq:no_energy} by a bootstrap argument. In particular, we define $T$ to be the largest time for which the following holds:
\begin{equation}\label{eq:no_energy_alt}
\begin{split}
J_1 (t) & = J_1 (0) + \O (\varepsilon^{\frac{5}{2}(1-\delta)+\kappa}),\\
| J_n (t)| & \lesssim \varepsilon^{2p} \qquad \forall\ n \in \N - \{ 1\},
\end{split}
\end{equation}
where $p>0$ will be fixed later in order to close the bootstrap argument.

In the case of $n=1$, we have that $J_1 (0)\leq c \varepsilon^{2(1-\delta)}$ while $\widehat{\Hc}_{4,3}$ contains at least three Fourier modes which are not $\pm 1$. Therefore $\{ J_1 ,  \widehat{\Hc}_{4,3} \} = \O (\varepsilon^{1-\delta+3p})$. Similarly, $\{J_1, \Rc_5\}=\O(\varepsilon^{5(1-\delta)})$. As a result, for all times $t\leq T$,
\[
|J_1 (t) - J_1 (0)|\lesssim \varepsilon^{5(1-\delta)}\, t + \varepsilon^{1-\delta+3p}\, t .
\]

For $n> 1$, we have that $J_n(0)=0$. One may show that $\{ J_n ,  \widehat{\Hc}_{4,3} \} = \O (\varepsilon^{3p+1-\delta})$ and $\{ J_n , \Rc_5\}=\O (\varepsilon^{p+4(1-\delta)})$. Thus for all $n>1$ and $t\leq T$,
\[
|J_n (t)| \lesssim \varepsilon^{p+4(1-\delta)}\, t +  \varepsilon^{3p+1-\delta}\, t .
\]

One quickly finds that in order to close the bootstrap argument we may fix $p=\frac{3}{2}(1-\delta)$ and require $t\ll \varepsilon^{-\frac{5}{2}(1-\delta)}$. It follows that $T\geq \varepsilon^{-\frac{5}{2}(1-\delta)+\kappa}$, as we wanted to prove. An analogous bootstrap argument (with an additional sum in $k$) shows that the solution $\norm{u(t)}_{Y_{\rho}} \lesssim \varepsilon^{1-\delta}$ for times $t\lesssim \varepsilon^{-\frac{5}{2}(1-\delta)+\kappa}$.

Finally, note that $\sum_{n=2}^{\infty} J_n (t)\leq \Hc_2 - J_1$. By \Cref{thm:normal_form}, \emph{(ii)},
\begin{equation}\label{eq:energy_tail}
\pa_t (\Hc_2 - J_1) = -\{J_1, \hat{\Hc}_{4,3}\}+ \{ \Hc_2-J_1 , \Rc_5\}= \{ \Hc_2-J_1 , \Rc_5\}+ \O (\varepsilon^{\frac{11}{2}(1-\delta)}) .
\end{equation}
Moreover, the Cauchy-Schwartz inequality and \Cref{thm:normal_form}, \emph{(vi)}, yield
\begin{equation}\label{eq:bound_tail}
\begin{split}
|\{ \Hc_2-J_1 , \Rc_5\}| & = \sum_{|k|>1} \Big | k^2 v_k \frac{\pa \Rc_5}{\pa v_k} - k^2 u_k \frac{\pa \Rc_5}{\pa u_k} \Big |\\
&  \lesssim \left(\sum_{|k|>1} k^2 |u_k(t)|^2 \right)^{1/2} \left( \sum_{|k|>1} k^2  |X_{\Rc_5} (u,v)_k|^2\right)^{1/2}\\
& \lesssim (\Hc_2 - J_1)^{1/2}\, \norm{X_{\Rc_5} (u,v)}_{Y_1} \lesssim \varepsilon^{4(1-\delta)}\, (\Hc_2 - J_1)^{1/2}.
\end{split}
\end{equation}
A simple bootstrap argument using \eqref{eq:energy_tail} and \eqref{eq:bound_tail} yields the desired bound
\[
\sum_{n\geq 2} J_n (t) \leq \Hc_2(t) - J_1(t) = \O ( \varepsilon^{3(1-\delta)+\kappa}).
\]
\end{proof}

Next, we show that $K_n$ in \eqref{eq:quantities} is also conserved for long times:

\begin{cor}\label{thm:conserved_K} Fix a small $\kappa>0$, $\delta\in (0,1)$ and suppose that $J_1(0)\leq c\varepsilon^{2(1-\delta)}$ for some $c>0$. Then there exists some $\varepsilon_0 = \varepsilon_0 (c,\delta)$ such that for all $\varepsilon<\varepsilon_0$ the following holds. If $t\lesssim \varepsilon^{-\frac{5}{2}(1-\delta)+\kappa}$, then we have that the solution to \eqref{eq:phi_beating} satisfies
\[
K_1(t) =  K_1 (0) + \O (\varepsilon^{\frac{5}{2}(1-\delta)+\kappa}).
\]
\end{cor}
\begin{proof}
By \Cref{thm:normal_form} and \Cref{thm:conserved_J},
\begin{equation}\label{eq:K1}
\pa_t K_1 = \{ K_1, \Hc_2 +\widehat{\Hc}_{4,1}+\widehat{\Hc}_{4,2}\}  + \O ( \varepsilon^{5(1-\delta)}).
\end{equation}

Note that $\{K_1,\Hc_2\}=0$, since $\{ K_1, J_1 \}=0$ and $\Hc_2 =J_1 + \sum_{|k|>1} k^2 u_k v_k$.
Moreover, denoting by $M$ the mass $\sum_k u_k v_k$,
\[
\{ K_1, \widehat{\Hc}_{4,1}\} = \{ K_1, 2K_1 (2M- J_1)\} = 2 K_1 \{ K_1 , 2M- J_1\} = 4K_1 \{ K_1, M\} =0
\]
since $M=J_1 +  \sum_{|k|>1} u_k v_k$ and $\{ K_1, J_1 \}=0$.

Finally, one can check that $\{ K_1, \widehat{\Hc}_{4,2}\} = 4i L_1 K_2$ with $L_n$ defined in \eqref{eq:quantities}. It follows from \eqref{eq:K1} that 
\[ 
\pa_t K_1  = 4i L_1 K_2 +  \O ( \varepsilon^{5(1-\delta)}).
\]
Using the fact that $|K_2|\leq J_2$ and $|L_1|\leq J_1$ together with \Cref{thm:conserved_J}, we find that $\pa_t K_1= \O ( \varepsilon^{5(1-\delta)})$ and the result follows by integrating in time.
\end{proof}

We are ready to derive an approximation for the dynamics of the beating NLS equation \eqref{eq:beating}.

\begin{prop}\label{thm:approximation} Fix a small $\kappa>0$, $\delta\in (0,1)$ and suppose that $|\alpha| + |\beta|\leq c\varepsilon^{-\delta}$ for some $c>0$. Then there exists some $\varepsilon_0 = \varepsilon_0 (c,\delta)$ such that for all $\varepsilon<\varepsilon_0$ the following holds. If $t\lesssim \varepsilon^{-\frac{5}{2}(1-\delta)+\kappa}$, then we have that the solution to \eqref{eq:beating} satisfies
\begin{equation*}
\begin{split}
u_1(t) & =\varepsilon \, e^{-it-4i \varepsilon^2 \mathrm{Re}( \alpha \overline{\beta})\, t} \, \left[\alpha \cos \left( 2\varepsilon^2 t\, (|\alpha|^2+ |\beta|^2)\right) - i \beta  \sin  \left( 2\varepsilon^2 t\, (|\alpha|^2+ |\beta|^2)\right) \right]+ \O ( \varepsilon^{1-\delta+2\kappa}) ,\\
u_{-1}(t) & = \varepsilon \, e^{-it-4i \varepsilon^2 \mathrm{Re}( \alpha \overline{\beta})\, t}\, \left[\beta \cos  \left( 2\varepsilon^2 t\, (|\alpha|^2+ |\beta|^2)\right) - i \alpha  \sin  \left( 2\varepsilon^2 t\, (|\alpha|^2+ |\beta|^2)\right) \right]  +  \O ( \varepsilon^{1-\delta+2\kappa}),  \\
\sum_{|k|\neq 1} |u_k(t)| & = \O(\varepsilon^{\frac{3}{2}(1-\delta)}).
\end{split}
\end{equation*}
\end{prop}
\begin{proof}
Consider $(u,\bar{u})=\Phi (w, \bar{w})$ with $\Phi$ given by \Cref{thm:normal_form}, in such a way that $u=(u_k)_{k\in\Z}$ are the Fourier coefficients of the solution to \eqref{eq:beating} and $w$ solves \eqref{eq:phi_beating} with initial data $w(0)=\Phi^{-1}\left( ( \varepsilon\,\alpha\, \delta_{j=1} + \varepsilon\,\beta\delta_{j=-1})_{j\in\Z}\right)$. 

By \Cref{thm:normal_form}, \emph{(vii)}, we have that
\[
\norm{w(0)}_{Y_0}\leq \norm{u(0)}_{Y_0} +  \O ( \norm{u(0)}_{Y_0}^2 )\lesssim \varepsilon^{1-\delta} +   \O ( \varepsilon^{2(1-\delta)})  \quad \Longrightarrow\quad  J_1 (0) \lesssim \varepsilon^{2(1-\delta)}.
\]

Let us obtain an approximation for $w$. By \Cref{thm:conserved_J} and \Cref{thm:conserved_K}
\[
\begin{split}
\pa_t w_1 & = \{ w_1, \widehat{\Hc}\} = \{ w_1, \widehat{\Hc}_2 +\widehat{\Hc}_{4,1} \}  + \O (\varepsilon^{4(1-\delta)}) \\
& = -i (1+2K_1)\, w_1 -2 i  (2 M -J_1)\,  w_{-1}  + \O (\varepsilon^{4(1-\delta)})  \\
&  = -i (1+2K_1(0))\, w_1 -2 i  J_1(0) \,  w_{-1}  +\O (\varepsilon^{\frac{7}{2}(1-\delta)+\kappa})
\end{split}
\]
since $M=J_1(0)$ due to the choice of initial data in \eqref{eq:phi_beating}.

A similar argument leads to
\[
\pa_t w_{-1} =  -i (1+2K_1(0))\, w_{-1} -2 i  J_1(0) \,  w_{1}  +\O (\varepsilon^{\frac{7}{2}(1-\delta)+\kappa}).
\]
Note that $J_1(0)=M$ and define the change of variables
\begin{equation}\label{eq:modulation}
\widetilde{w_k}(t) = \exp \left(i t +2i K_1(0) t \right) w_k(t) \qquad \mbox{for } k=\pm 1.
\end{equation}
In these new coordinates, we obtain:
\[
\begin{cases}
\pa_t \widetilde{w_1} = -2iJ_1(0) \widetilde{w_{-1}}+ \O (\varepsilon^{\frac{7}{2}(1-\delta)+\kappa}) \\
\pa_t \widetilde{w_{-1}} = -2iJ_1(0) \widetilde{w_{1}} + \O (\varepsilon^{\frac{7}{2}(1-\delta)+\kappa})
\end{cases} 
\]
Writing the system in terms of $\widetilde{w_1}+ \widetilde{w_{-1}}$ and $\widetilde{w_1}- \widetilde{w_{-1}}$ and integrating in time yields the following explicit solution:
\[
\begin{cases}
\widetilde{w_1} (t) = w_1 (0) \cos (2J_1(0)t) - i w_{-1}(0) \sin (2J_1(0)t) + \O (\varepsilon^{1-\delta+2\kappa}) \\
\widetilde{w_{-1}}(t) = w_{-1} (0) \cos (2J_1(0)t) - i w_1(0) \sin (2J_1(0)t)+\O (\varepsilon^{1-\delta+2\kappa}) .
\end{cases}
\]
Undoing the transformation \eqref{eq:modulation} yields:
\[
\begin{cases}
w_1(t) (t) = e^{-i t - 2i K_1(0) t} \, \left[ w_1 (0) \cos (2J_1(0)t) - i w_{-1}(0) \sin (2J_1(0)t)\right] + \O (\varepsilon^{1-\delta+2\kappa}) \\
w_{-1}(t) =e^{-i t - 2i K_1(0) t} \, \left[ w_{-1} (0) \cos (2J_1(0)t) - i w_1(0) \sin (2J_1(0)t)\right] +\O (\varepsilon^{1-\delta+2\kappa}) .
\end{cases}
\]

By \Cref{thm:normal_form}, \emph{(vii)}, we have that $u_k(t)= w_k(t) + \O( \varepsilon^{2(1-\delta)})$, and therefore:
\[
\begin{cases}
u_1 (t) = e^{-i t - 2i K_1(0) t} \,\left[ u_1 (0) \cos (2J_1(0)t) - i u_{-1}(0) \sin (2J_1(0)t) \right] + \O (\varepsilon^{1-\delta+2\kappa}) \\
u_{-1} (t) = e^{-i t - 2i K_1(0) t} \,\left[  u_{-1} (0) \cos (2J_1(0)t) - i u_1(0) \sin (2J_1(0)t)\right]+\O (\varepsilon^{1-\delta+2\kappa}) .
\end{cases}
\]
The desired formulas follow from a final application of \Cref{thm:normal_form}, \emph{(vii)}, which guarantees
\[
\begin{split}
K_1(0) & =2\mbox{Re}\left(w_1 (0) \, \overline{w_{-1} (0) }\right)= 2\varepsilon^2\, \mbox{Re}(\alpha \overline{\beta})+  \O( \varepsilon^{3(1-\delta)})\\
J_1(0) & =|w_1 (0)|^2+ |w_{-1} (0)|^2 =\varepsilon^2\, (|\alpha|^2 + |\beta|^2) +  \O( \varepsilon^{3(1-\delta)}).
\end{split}
\]
The bound on  $\sum_{|k|\neq 1} |u_k(t)|$ follows from \Cref{thm:normal_form}, \emph{(vii)}, with $\rho=0$ and \eqref{eq:energy_tail}-\eqref{eq:bound_tail}, in view of the inequality:
\[
\left( \sum_{|k|\neq 1} |w_k(t)| \right)^2 \lesssim \sum_{|k|\neq 1} |k|^2 |w_k(t)|^2 \lesssim \varepsilon^{3(1-\delta)}.
\]
\end{proof}

\section{Large deviations principle}\label{sec:LDP}
Consider \eqref{eq:beating} with $\alpha$ and $\beta$ two complex, independent Gaussian random variables with $\E \alpha = \E \beta =0$, and with possibly different variance:
\begin{equation}
\begin{cases}
i \pa_t u + \pa_x^2 u = 2\, \cos (2x)\, |u|^2 u, \quad x\in\T,\\
u(t,x)\mid_{t=0}= \varepsilon\,  ( \alpha\, e^{ix} + \beta\, e^{-ix}).
\end{cases}
\end{equation}

Writing $u(t,x)=\sum_{k\in\Z} u_k(t) e^{ikx}$, \Cref{thm:approximation} yields an approximation to the solution for long timescales $|t|\leq \varepsilon^{-\frac{5}{2}(1-\delta)+\kappa}$ as long as our initial data is not too large. Under this approximation, we have that 
\[
\begin{split}
\sup_{x\in\T} |u(t,x)|^2 &= \sup_{x\in\T} \Big |\sum_{k\in\Z}  u_k(t) e^{ikx}\Big |^2 = \sup_{x\in\T} \Big | u_1(t) \, e^{ix} +u_{-1}(t) e^{-ix} \Big |^2 + \O( \varepsilon^{3(1-\delta)}) \\
& = |u_{1}(t)|^2 + |u_{-1}(t)|^2 + 2 \sup_{x\in\T} \, \mbox{Re} ( u_{1}(t)\, \overline{u_{-1}(t)}\, e^{2ix} ) + \O( \varepsilon^{3(1-\delta)})\\
& =    \left( |u_{1}(t)|^2 + |u_{-1}(t)|^2 + 2 |u_{1}(t)| |u_{-1}(t)|\right) +\O( \varepsilon^{3(1-\delta)})\\
& =( |u_{1}(t)|+|u_{-1}(t)|)^2 + \O( \varepsilon^{3(1-\delta)})
\end{split}
\]
and thus 
\begin{equation}\label{eq:sup}
\begin{split}
\sup_{x\in\T} |u(t,x)| =&\ \varepsilon\, \sqrt{|\alpha|^2 \cos^2 \left( 2\varepsilon^2 t\, (|\alpha|^2+ |\beta|^2)\right) + |\beta|^2  \sin^2 \left( 2\varepsilon^2 t\, (|\alpha|^2+ |\beta|^2)\right) } \\
& + \varepsilon\, \sqrt{|\beta|^2 \cos^2 \left( 2\varepsilon^2 t\, (|\alpha|^2+ |\beta|^2)\right) + |\alpha|^2  \sin^2 \left( 2\varepsilon^2 t\, (|\alpha|^2+ |\beta|^2)\right)} + \O( \varepsilon^{\frac{3}{2}(1-\delta)}).
\end{split}
\end{equation}

In order to obtain sharp bounds on the tails of the supremum of $u$, we will use an elementary large deviations principle for the $\ell^1$ and $\ell^2$ norms of a vector of two independent complex Gaussian random variables. A generalization of \eqref{eq:LDP_L1} below to the case of an infinite vector may be found in \cite{GGKS}.
\begin{lem}\label{thm:L1_L2_LDP} Let $\alpha$ and $\beta$ be complex Gaussian random variables with zero mean and variance $\E |\alpha|^2 =\sigma_{\alpha}^2$, $\E |\beta|^2 = \sigma_{\beta}^2$. We then have
\begin{align}
\lim_{\lambda \rightarrow \infty} \lambda^{-2}\, \log \P (  |\alpha| + |\beta|  > z_0 \lambda ) & = - \frac{z_0^2}{\sigma_{\alpha}^2+ \sigma_{\beta}^2} , \label{eq:LDP_L1} \\
\lim_{\lambda \rightarrow \infty} \lambda^{-2}\, \log \P (  \sqrt{2} \sqrt{|\alpha|^2 + |\beta|^2}   > z_0 \lambda ) & = - \frac{z_0^2}{2\,\max\{\sigma_{\alpha}^2, \sigma_{\beta}^2\}} .\label{eq:LDP_L2}
\end{align}
In particular, the right-hand side of \eqref{eq:LDP_L1} and \eqref{eq:LDP_L2} are equal when $\sigma_{\alpha}^2 =\sigma_{\beta}^2$.
\end{lem}
\begin{proof}
{\bf $\ell^1$-norm.} If $\eta$ is a complex Gaussian r.v.\ with $\E \eta=0$ and $\E |\eta|^2 = \sigma^2$ then $|\eta|$ is Rayleigh distributed with probability density function (pdf):
\[
f(a)= \frac{2a}{\sigma^2} \, \exp\left(-\frac{a^2}{\sigma^2}\right) \qquad \mbox{for}\ a\geq 0.
\]
As a result,
\[
\begin{split}
 \P (  |\alpha| + |\beta|   > z_0 \lambda )  =&\  \int_{0}^{\infty}\int_0^{\infty} \chi_{a+b>z_0\lambda} \, \frac{4ab}{\sigma_{\alpha}^2\,\sigma_{\beta}^2} \,\exp\left(-\frac{a^2}{\sigma_{\alpha}^2}-\frac{b^2}{\sigma_{\beta}^2} \right) \, da\, db \\
  = &\ \int_0^{z_0\lambda} \frac{2a}{\sigma_{\alpha}^2} \, \exp\left(-\frac{a^2}{\sigma_{\alpha}^2} - \frac{(z_0\lambda-a)^2}{\sigma_{\beta}^2}\right)\, da +  \int_{z_0\lambda}^{\infty} \frac{2a}{\sigma_{\alpha}^2} \, \exp\left(-\frac{a^2}{\sigma_{\alpha}^2}\right)\, da\\
  = &\ \exp\left(-\frac{z_0^2\lambda^2}{\sigma_{\alpha}^2+\sigma_{\beta}^2} \right)\, \Ic (\lambda) +  \exp\left(-\frac{z_0^2\, \lambda^2}{\sigma_{\alpha}^2}\right) ,
 \end{split}
\]
where 
\begin{equation}
\Ic (\lambda)= \int_0^{z_0\lambda} \frac{2a}{\sigma_{\alpha}^2} \, \exp\left(-\frac{1}{\sigma_{\alpha}^2\sigma_{\beta}^2} \, \left( \sqrt{\sigma_{\alpha}^2 + \sigma_{\beta}^2} \, a - \frac{z_0\lambda\sigma_{\alpha}^2}{\sqrt{\sigma_{\alpha}^2 + \sigma_{\beta}^2}}\right)^2 \right)\, da.
\end{equation}

In order to obtain \eqref{eq:LDP_L1}, it suffices to show that $\lambda \lesssim \Ic ( \lambda) \lesssim \lambda^2$ for $\lambda\gg 1$. Note that bounding the exponential term by 1 yields $\Ic (\lambda)\leq z_0^2 \lambda^2 /\sigma_{\alpha}^2$.

To derive the lower bound, we consider $\lambda\gg z_0^{-1}$ sufficiently large, and we write:
\[
\begin{split}
\Ic (\lambda) & \geq \int_{\frac{z_0\lambda\sigma_{\alpha}^2}{\sigma_{\alpha}^2 + \sigma_{\beta}^2}}^{z_0\lambda} \frac{2a}{\sigma_{\alpha}^2} \, \exp\left(-\frac{1}{\sigma_{\alpha}^2\sigma_{\beta}^2} \, \left( \sqrt{\sigma_{\alpha}^2 + \sigma_{\beta}^2} \, a - \frac{z_0\lambda\sigma_{\alpha}^2}{\sqrt{\sigma_{\alpha}^2 + \sigma_{\beta}^2}}\right)^2 \right)\, da\\
& \gtrsim \lambda\, \int_{\frac{z_0\lambda\sigma_{\alpha}^2}{\sigma_{\alpha}^2 + \sigma_{\beta}^2}}^{z_0\lambda} \exp\left(-\frac{1}{\sigma_{\alpha}^2\sigma_{\beta}^2} \, \left( \sqrt{\sigma_{\alpha}^2 + \sigma_{\beta}^2} \, a - \frac{z_0\lambda\sigma_{\alpha}^2}{\sqrt{\sigma_{\alpha}^2 + \sigma_{\beta}^2}}\right)^2 \right)\, da\\
& \gtrsim  \lambda \left[ \int_0^{\infty} \exp\left(-\frac{r^2}{\sigma_{\alpha}^2\sigma_{\beta}^2}\right) \, dr - \int_{\frac{z_0\lambda\sigma_{\beta}^2}{\sqrt{\sigma_{\alpha}^2+\sigma_{\beta}^2}}}^{\infty} \exp\left(-\frac{r^2}{\sigma_{\alpha}^2\sigma_{\beta}^2}\right) \, dr\right] \gtrsim \lambda.
\end{split}
\]

\noindent {\bf $\ell^2$-norm.}
First we assume that $\sigma_{\alpha}^2 =\sigma_{\beta}^2=\sigma^2$. Then the r.v.\ $|\alpha|^2$ has exponential distribution with pdf
\[
f_{|\alpha|^2}(a)= \sigma^{-2}\, e^{-a/\sigma^2} \qquad \mbox{for}\ a\geq 0.
\]
As a result, the pdf of $|\alpha|^2 + |\beta|^2$ is
\[
f(a) = \int_0^{a} \sigma^{-4}\, e^{-b/\sigma^2} e^{-(a-b)/\sigma^2}\, db = \frac{a}{\sigma^{4}} e^{-a/\sigma^2}\qquad \mbox{for}\ a\geq 0.
\]
Integration by parts leads to the following equality:
\[
\begin{split}
\P (  \sqrt{2} \sqrt{|\alpha|^2 + |\beta|^2}   > z_0 \lambda ) & = \frac{z_0^2\lambda^2}{2\sigma^2} \exp\left( -\frac{z_0^2 \lambda^2}{2\sigma^2}\right) + \int_{z_0^2 \lambda^2/2}^{\infty}  \frac{1}{\sigma^2} e^{-a/\sigma^2}\, da \\
&= \left( \frac{z_0^2\lambda^2}{2\sigma^2} +1\right) \, \exp\left( -\frac{z_0^2 \lambda^2}{2\sigma^2}\right) ,
\end{split}
\]
which directly implies \eqref{eq:LDP_L2}.

Assume next that  $\sigma_{\alpha}^2 > \sigma_{\beta}^2$. Then the pdf of $|\alpha|^2 + |\beta|^2$ is
\[
f(a) = \int_0^{a} \sigma_{\alpha}^{-2} \sigma_{\beta}^{-2}\, \exp\left( -\frac{b}{\sigma_{\beta}^2} -\frac{a-b}{\sigma_{\alpha}^2}\right)\, db 
= \frac{1}{\sigma_{\alpha}^2-\sigma_{\beta}^2} \,  \left(\exp\left(-a/\sigma_{\alpha}^2\right) - \exp\left(-a/\sigma_{\beta}^2\right) \right)  \quad \mbox{for}\ a\geq 0.
\]
Direct integration yields
\[
\begin{split}
\P (  \sqrt{2} \sqrt{|\alpha|^2 + |\beta|^2}   > z_0 \lambda ) & = \frac{1}{\sigma_{\alpha}^2-\sigma_{\beta}^2} \left[ \sigma_{\alpha}^2 \,\exp\left( -\frac{z_0^2 \lambda^2}{2\sigma_{\alpha}^2}\right) - \sigma_{\beta}^2 \exp\left( -\frac{z_0^2 \lambda^2}{2\sigma_{\beta}^2}\right)\right]
\end{split}
\]
and \eqref{eq:LDP_L2} follows.
\end{proof}

\subsection{Equal variance}

In the case of resonant modes with equal variance which exchange energy, we may show that a time-independent large deviations principle holds. In order to obtain such an LDP, we derive sharp upper and lower bounds.

For the purpose of obtaining an upper bound which is independent of time, we may exploit the following inequalities:
\begin{equation}\label{eq:general_upper_bound}
\begin{split}
\sup_{x\in\T} |u(t,x)| & \leq |u_1(t)| + |u_{-1}(t)| + \O( \varepsilon^{\frac{3}{2}(1-\delta)}) \leq \sqrt{2}\, \sqrt{|u_1(t)|^2 + |u_{-1}(t)|^2}+ \O( \varepsilon^{\frac{3}{2}(1-\delta)}) \\
&  = \sqrt{2} \varepsilon\, \sqrt{|\alpha|^2 + |\beta|^2} + \O( \varepsilon^{\frac{3}{2}(1-\delta)}) \qquad \forall t\geq 0,
\end{split}
\end{equation}
assuming $|\alpha|+|\beta|\leq c\, \varepsilon^{-\delta}$.

In order to obtain a lower bound, we will exploit the following inequality instead:
\begin{equation}\label{eq:general_lower_bound}
\varepsilon\, (|\alpha| + |\beta|) \leq\sup_{x\in\T} |u(t,x)| +  \O( \varepsilon^{\frac{3}{2}(1-\delta)}),
\end{equation}
assuming $|\alpha|+|\beta|\leq c\, \varepsilon^{-\delta}$. This inequality follows from \eqref{eq:sup} by differentiating in time and showing that the minimum is achieved exactly when the sine or the cosine vanish.

These results give rise to the following LDP for solutions to \eqref{eq:beating} where the random variables in the initial data have the same variance:

\begin{prop}\label{thm:main_equal_var} Fix $z_0>0$, a small $\kappa>0$ and $\delta\in (0,1)$. Let $u$ be the solution to \eqref{eq:beating} where $\alpha$ and $\beta$ are complex, i.i.d. Gaussian random variables with zero mean and variance $\E |\alpha|^2 = \E |\beta|^2 = \sigma^2$. Then for times $t\lesssim \varepsilon^{-\frac{5}{2}(1-\delta)+\kappa}$, we have that
\[
\lim_{\varepsilon\rightarrow 0^{+}} \varepsilon^{2\delta} \log \P \left( \sup_{x\in\T} |u(t,x)| \geq z_0 \varepsilon^{1-\delta}\right) = - \frac{z_0^2}{2\sigma^2}.
\]
\end{prop}
\begin{proof}
Let $\varepsilon$ be sufficiently small for \Cref{thm:approximation} to hold. Define the events
\[
\Ac (\lambda) = \{  \norm{u(t)}_{L^{\infty}_x (\T)} \geq  \lambda\},\quad \Bc_1 ( \lambda ) =\{ |\alpha| +|\beta| \geq  \lambda \},\quad  \Bc_2 (\lambda ) = \{ \sqrt{|\alpha|^2 + |\beta|^2} \geq  \lambda\}.
\]

By \eqref{eq:LDP_L1}, we have that 
\begin{equation}\label{eq:fix_c}
\lim_{\varepsilon\rightarrow 0^{+}} \varepsilon^{2\delta}\, \log \P ( \Bc_1 (c\varepsilon^{-\delta}) ) \leq -\frac{c^2}{2\sigma^2} < - 100 \, \frac{z_0^2}{2\sigma^2}
\end{equation}
by fixing $c> 10 z_0$.

Note that 
\[
\P(\Ac ( z_0 \varepsilon^{1-\delta})) = \P(\Ac ( z_0 \varepsilon^{1-\delta}) \cap  \Bc_1 (c\varepsilon^{-\delta})^c ) +  \P(\Ac ( z_0 \varepsilon^{1-\delta}) \cap  \Bc_1 (c\varepsilon^{-\delta}) ) .
\]
By \eqref{eq:fix_c}, the second summand is negligible since
\[
0 \leq \P(\Ac ( z_0 \varepsilon^{1-\delta}) \cap  \Bc_1 (c\varepsilon^{-\delta}) ) \leq \P ( \Bc_1 (c\varepsilon^{-\delta}) ) \leq \exp\left( - 100 \, \frac{z_0^2\varepsilon^{-2\delta}}{2\sigma^2}\right).
\]
We highlight that our equation \eqref{eq:beating} is a.s.\ globally well-posed and thus the solution is well-defined in the set $\Bc_1 (c\varepsilon^{-\delta})$.

Finally, \eqref{eq:general_upper_bound} and \eqref{eq:LDP_L2} yield
\begin{equation}\label{eq:step_L2}
 \log \P(\Ac ( z_0 \varepsilon^{1-\delta}) \cap  \Bc_1 (c\varepsilon^{-\delta})^c ) \leq \log \P ( \Bc_2 ( \frac{z_0 \varepsilon^{-\delta}}{\sqrt{2}} - \O ( \varepsilon^{\frac{1}{2}-\frac{3}{2}\delta}))) \leq -\frac{z_0^2}{2\sigma^2}\, \varepsilon^{-2\delta} + o( \varepsilon^{-2\delta}).
\end{equation}
Similarly, \eqref{eq:general_lower_bound} and \eqref{eq:LDP_L1} yield:
\begin{equation}
 \log \P(\Ac ( z_0 \varepsilon^{1-\delta}) \cap  \Bc_1 (c\varepsilon^{-\delta})^c ) \geq  \log \P ( \Bc_1 ( z_0 \varepsilon^{-\delta}+ \O ( \varepsilon^{\frac{1}{2}-\frac{3}{2}\delta}))) \geq  -\frac{z_0^2}{2\sigma^2}\, \varepsilon^{-2\delta} - o( \varepsilon^{-2\delta})
\end{equation}
and the desired limit follows.
\end{proof}

\subsection{Different variance}

The rest of this article is devoted to the richer case of initial data of the form \eqref{eq:beating} where $\alpha$ and $\beta$ are independent, complex Gaussian random variables with zero mean and \emph{different} variances 
\[
\E |\alpha|^2 = \sigma_{\alpha}^2 >  \sigma_{\beta}^2 = \E |\beta|^2.
\]

\subsubsection{Sub-resonant timescales}

Our first result concerns timescales $t\ll \varepsilon^{-2(1-\delta)}$ where the dynamics are well-approximated by the linear equation. Indeed, suppose that $d(\varepsilon):= \varepsilon^{2(1-\delta)} t = o(1)$ as $\varepsilon\rightarrow 0^{+}$.

Then we may carry out another approximation step in \eqref{eq:sup}. Assuming that $|\alpha|+|\beta|\leq c\varepsilon^{-\delta}$ for some $c>0$, we have that
\[
\cos \left( 2\varepsilon^2 t\, (|\alpha|^2+ |\beta|^2)\right) = 1 + \O ( d(\varepsilon)^2), \qquad \sin \left( 2\varepsilon^2 t\, (|\alpha|^2+ |\beta|^2)\right) = \O( d(\varepsilon)).
\]
As a result, \eqref{eq:sup} becomes
\begin{equation}\label{eq:sup_subresonant}
\sup_{x\in\T} |u(t,x)| = \varepsilon\, (|\alpha|+|\beta|) + \O( \varepsilon d(\varepsilon) (|\alpha|+|\beta|)) + \O( \varepsilon^{\frac{3}{2}(1-\delta)}).
\end{equation}

This leads to
 
\begin{prop}\label{thm:LDP_diff_linear} Fix $z_0>0$ and $\delta\in (0,1)$. Let $u$ be the solution to \eqref{eq:beating} where $\alpha$ and $\beta$ are complex, independent Gaussian random variables with zero mean and variance $\E |\alpha|^2 = \sigma_{\alpha}^2 >  \sigma_{\beta}^2 = \E |\beta|^2$. Then for times $t$ such that $\lim_{\varepsilon\rightarrow 0^{+}} \varepsilon^{2(1-\delta)} t = 0$, we have that
\[
\lim_{\varepsilon\rightarrow 0^{+}} \varepsilon^{2\delta} \log \P \left( \sup_{x\in\T} |u(t,x)| \geq z_0 \varepsilon^{1-\delta}\right) = - \frac{z_0^2}{\sigma_{\alpha}^2+\sigma_{\beta}^2}.
\]
\end{prop}
\begin{proof}
The proof is identical to that of \Cref{thm:main_equal_var}, the only difference being that \eqref{eq:step_L2} may be replaced by another application of \eqref{eq:LDP_L1} in view of \eqref{eq:sup_subresonant}.
\end{proof}

\subsubsection{Resonant timescales}

Our next result is the analysis of the tails of the distribution of the maximum for times $t= \tau \varepsilon^{-2(1-\delta)}$ for some $\tau\in (0,\infty)$. In this regime, we see a departure from the ``linear'' LDP obtained in \Cref{thm:LDP_diff_linear} due to strong resonances which cause a exchange of energy between the two main Fourier modes.

The best strategy to increase the probability of seeing an extreme wave is to concentrate most of the energy in the Fourier mode with the largest variance. However, the nonlinear dynamics occasionally obstruct this concentration. For long times, we will see that this mechanism is always possible.

\begin{thm}\label{thm:resonantLDP_diff_variance}
Fix $z_0>0$ and $\delta\in (0,1)$. Let $u$ be the solution to \eqref{eq:beating} where $\alpha$ and $\beta$ are complex, independent Gaussian random variables with zero mean and variance $\E |\alpha|^2 = \sigma_{\alpha}^2 >  \sigma_{\beta}^2 = \E |\beta|^2$. Then there exists a function $\Jc$ such that, for $t= \tau \varepsilon^{-2(1-\delta)}$ with $\tau\in (0,\infty)$,
\begin{equation}\label{eq:resonantLDP_diff_variance}
\liminf_{\varepsilon\rightarrow 0} \varepsilon^{2\delta} \log \P \left( \sup_{x\in\T} |u(t,x)| \geq z_0 \varepsilon^{1-\delta}\right) \geq - \frac{\Jc (z_0,\tau)^2}{\sigma_{\alpha}^2},
\end{equation}
where $\Jc (z_0,\tau)$ is a c\`adl\`ag function which takes values in the interval $[z_0/\sqrt{2}, z_0]$ and has an infinite countable number of isolated jump discontinuities $(\tau_j^{\infty}(z_0))_{j\in\N}$. When $\tau$ is a point of continuity, $\Jc (z_0,\tau)$ is the minimal solution to the following implicit equation
\begin{equation}\label{eq:implicit_J}
z_0= \Jc (z_0,\tau)\, \left( |\cos\left(2\tau\, \Jc (z_0,\tau)^2\right)| + |\sin\left(2\tau \,\Jc (z_0,\tau)^2\right)|\right).
\end{equation}
When $\tau$ is a jump point, $\Jc (z_0,\tau):= \lim_{s\rightarrow \tau^{+}} \Jc (z_0,s)$.

Moreover, 
\begin{itemize}
\item $\Jc (z_0,\tau)$ is continuous for $\tau \in (0, \frac{\pi}{4\,z_0^2})$ and
\begin{equation}\label{eq:tau_zero_limit}
\lim_{\tau\rightarrow 0^{+}} \Jc (z_0,\tau) = z_0.
\end{equation}
\item The following limit exists
\begin{equation}\label{eq:tau_infty_limit}
\lim_{\tau\rightarrow +\infty} \Jc (z_0,\tau)=\frac{z_0}{\sqrt{2}}
\end{equation}
\item The jump points of $\Jc$, $(\tau_j^{\infty}(z_0))_{j\in\N}\subset ( \frac{\pi}{4\,z_0^2},\infty)$, are increasing in $j\in\N$ and they satisfy the following asymptotic formula:
\begin{equation}\label{eq:asymp_jumps}
\tau_j^{\infty}(z_0) = \frac{\pi}{2\,z_0^2}\, \left( j - \frac{1}{2}\right) + \O\left( \frac{1}{j\, z_0^2}\right)  \qquad\mbox{as}\ j\rightarrow\infty.
\end{equation}
\end{itemize}
\end{thm}

\begin{rk} We highlight that we do not expect \eqref{eq:resonantLDP_diff_variance} to be sharp. Indeed, for general $t$ the lower bound coming from \eqref{eq:general_lower_bound} and \eqref{eq:LDP_L1} yields:
\begin{equation}
\liminf_{\varepsilon\rightarrow 0} \varepsilon^{2\delta} \log \P \left( \sup_{x\in\T} |u(t,x)| \geq z_0 \varepsilon^{1-\delta}\right) \geq - \frac{z_0^2}{\sigma_{\alpha}^2+\sigma_{\beta}^2}.
\end{equation}
This is sharper than \eqref{eq:resonantLDP_diff_variance} for small $\tau>0$ in view of \eqref{eq:tau_zero_limit}. However, we do expect \eqref{eq:resonantLDP_diff_variance} to be sharp as $\tau\rightarrow\infty$ thanks to \eqref{eq:tau_infty_limit} and the upper bound given by \eqref{eq:general_upper_bound} and \eqref{eq:LDP_L2}.
\end{rk}

\Cref{sec:implicit} is devoted to the study of the function $\Jc (z_0,\tau)$, including the proofs of \eqref{eq:implicit_J}-\eqref{eq:asymp_jumps}. The proof of \eqref{eq:resonantLDP_diff_variance} is the content of \Cref{thm:LDP_resonant} and \Cref{thm:LDP_resonant2}.

\subsubsection{Super-resonant times}

Our next result is the analysis of the tails of the distribution of the maximum for times $t=t(\varepsilon)\gg  \varepsilon^{-2(1-\delta)}$, i.e. $\lim_{\varepsilon\rightarrow 0^{+}} t(\varepsilon) \varepsilon^{2(1-\delta)} = \infty$. For all such timescales, the tails of the distribution stabilize around a new constant value larger than the one obtained in the sub-resonant regime $t\ll \varepsilon^{-2(1-\delta)}$.

\begin{thm}\label{thm:super_resonantLDP_diff_variance}
Fix $z_0>0$,  $\delta\in (0,1)$ and $\kappa>0$ as small as desired. Let $u$ be the solution to \eqref{eq:beating} where $\alpha$ and $\beta$ are complex, independent Gaussian random variables with zero mean and variance $\E |\alpha|^2 = \sigma_{\alpha}^2 >  \sigma_{\beta}^2 = \E |\beta|^2$. Suppose that $ \varepsilon^{-2(1-\delta)}\ll t\lesssim \varepsilon^{-\frac{5}{2}(1-\delta)+\kappa}$. Then 
\begin{equation}\label{eq:super_resonantLDP_diff_variance}
\lim_{\varepsilon\rightarrow 0} \varepsilon^{2\delta} \log \P \left( \sup_{x\in\T} |u(t,x)| \geq z_0 \varepsilon^{1-\delta}\right) = -\frac{z_0^2}{2\sigma_{\alpha}^2}.
\end{equation}
\end{thm}

We postpone the proof of this theorem to \Cref{sec:super_res_LDP}.

\section{Analysis of the implicit equation}\label{sec:implicit}

\subsection{Existence of solutions}

For fixed $\lambda>0$,  consider the following implicit equation for $(\tau,y)$:
\begin{equation}\label{eq:def_h}
\lambda= y\, h(2\tau y^2) \qquad \mbox{where} \qquad h(\xi) := |\cos (\xi)| + |\sin (\xi)|. 
\end{equation}

In order to show that a solution to \eqref{eq:def_h} exists, we rewrite this as the level set $F(\tau,y)=0$ for 
\begin{equation}\label{eq:def_F}
F(\tau, y)= y\, h(2\tau y^2)-\lambda.
\end{equation}

\begin{lem}\label{thm:existence} Fix $\lambda>0$, then the following hold:
\begin{enumerate}[(i)]
\item For each $\tau\geq 0$, there exists at least one $y=y (\tau)$ such that 
\[
F(\tau,y(\tau))=0.
\]
\item For each $\tau\geq 0$, all solutions $y$ to $F(\tau,y)=0$ must live in the interval $\left[ \lambda/\sqrt{2},\lambda\right]$.
\item There exists some $\rho>0$ such that the solution $y(\tau)$ is unique and differentiable in $\tau$ for all $\tau\in (0,\rho)$ and Lipschitz for $\tau \in (-\rho,\rho)$. Moreover 
\[
\lim_{\tau\rightarrow 0^{+}} y(\tau)= \lambda.
\]
\end{enumerate}
\end{lem}
\begin{proof}
{\bf \emph{ (i)}} Note that the function $h$ in \eqref{eq:def_h} ranges between $1$ and $\sqrt{2}$. Therefore
\begin{equation}\label{eq:F_Bolzano_up}
F(\tau,\lambda)= \lambda h(2\tau \lambda^2) - \lambda\geq 0,
\end{equation}
while
\begin{equation}\label{eq:F_Bolzano_low}
F\left(\tau,\frac{\lambda}{\sqrt{2}}\right)= \frac{\lambda}{\sqrt{2}}\,  h(\tau \,\lambda^2) - \lambda\leq 0.
\end{equation}
Given that $F$ is continuous in $y$, there must exist at least one $y\in \left[ \frac{\lambda}{\sqrt{2}} ,\lambda\right]$ such that $F(\tau,y)=0$.

{\bf \emph{(ii)}}  Following the same reasoning, if $y< \lambda/\sqrt{2}$ then \eqref{eq:F_Bolzano_low} is a strict inequality.
An analogous reasoning shows that $y> \lambda$ cannot be a solution.

{\bf \emph{(iii)}} By plugging in $\tau=0$ in \eqref{eq:def_h}, we quickly find that $F(0,y)=y-\lambda=0$. This implies that $y=\lambda$. 

Next notice that the function $h(\xi)$ in \eqref{eq:def_h} is differentiable as long as $\xi\neq\frac{j\,\pi}{2}$ for $j\in \Z$.
However, $h$ is certainly Lipschitz, and the right-derivative and left-derivatives of $h$ do exist at such points. Note that
\[
\frac{\pa F}{\pa y} (\tau, y)= h(2\tau y^2)+h'(2\tau y^2)\, 4\tau\, y^2,
\]
for any $\tau\neq 0$ in a small neighborhood of $(\tau,y)=(0,\lambda)$. In particular, one may take $\tau\rightarrow 0$ and show that $\frac{\pa F}{\pa y} (0, \lambda) =1$ exists.

By the Implicit Function theorem for Lipschitz functions \cite{Clarke1,Clarke2}, there exists a neighborhood of $(\tau,y)=(0,\lambda)$ where the solution curve to $F(\tau,y)=0$ can be uniquely expressed as $F(\tau,y(\tau))$ for a Lipschitz function $y(\tau)$.
\end{proof}

Despite the solution to $F(\tau,y)=0$ admitting a parametrization $y=y(\tau)$ for small $\tau$, this parametrization is not global in $\tau$ as $y'$ blows up quickly. That being said, the $2-$dimensional curve $F(\tau,y)=0$ admits a nice global parametrization with respect to slightly different variables.

Let $\xi=2\,  \tau\, y^2$. Let us introduce the curve:
\begin{equation}\label{eq:def_F2}
\widetilde{F} (\tau,\xi)= \xi \, h(\xi)^2 -  2\tau\, \lambda^2.
\end{equation}
We note that the curve $\widetilde{F} (\tau,\xi)=0$ is equivalent to $F(\tau,y)=0$ for all $\tau,y\in (0,\infty)$. Moreover, we have the following result:

\begin{lem}\label{thm:curve_tau_xi} Fix $\lambda>0$. The curve $\{ (\tau,\xi)\in (0,\infty)^2 \mid \widetilde{F} (\tau,\xi)=0\}$ admits a global parametrization
\[
\begin{split}
\xi \in (0,\infty)\quad  \longmapsto \quad \tau(\xi)= \frac{\xi \, h(\xi)^2}{2\,\lambda^2}.
\end{split}
\]
which is locally Lipschitz and has Lipschitz local inverse. In particular, this curve (and thus $F=0$) is connected and locally Lipschitz.
\end{lem}
\begin{proof}
The parametrization comes from explicitly solving $\widetilde{F}=0$. The Lipschitz property follows from that of $h$. Moreover, $\tau((0,\infty))$ must be connected by the continuity of $\tau(\xi)$. These properties are inherited by $F=0$ thanks to the fact that the change of coordinates $(\tau,\xi)\mapsto (\tau,y)=(\tau, \sqrt{\xi /(2\tau)})$ is a $C^1$ bijection with $C^1$ inverse in the domain $\tau,\xi>0$.
\end{proof}

\subsection{Analysis of multiple solutions}

While the curve $\widetilde{F} (\tau,\xi)=0$ behaves well, there are various issues when we try to view it under $(\tau,y(\tau))$ coordinates due to the fact that $y(\tau)$ is multi-valued at various points (multiple solutions to $F(\tau,y)=0$ exist for fixed $\tau$).

Some of these points correspond to values of $\xi$ where $h$ is not differentiable (only Lipschitz). At such points two new solutions $y(\tau)$ are ``born'' with different initial derivatives. The rest of these points correspond to values where $y'$ blows up for a solution curve. These correspond to collisions between two solutions which cease to exist at that point. In particular, the function $\Jc (\lambda,\tau)$ introduced in \eqref{thm:resonantLDP_diff_variance} will have jumps at such points.

In order to justify this description, let us introduce the following sequence of times. For each $j\in\N$, let
\begin{equation}\label{eq:def_tau}
\tau_j := \frac{\pi}{2\,\lambda^2}\, \left( j - \frac{1}{2}\right), \qquad\qquad \widetilde{\tau}_j := \frac{\pi}{4\,\lambda^2}\,j\ .
\end{equation}
Note that $\tau_j =  \widetilde{\tau}_{2j-1}$ and $\widetilde{\tau}_{2j}=(\tau_j+\tau_{j+1})/2$ for all $j\in\N$.

We have the following result:

\begin{prop}\label{thm:bifurcations} Consider the equation $F(\tau,y)=0$ in \eqref{eq:def_F} for $\tau>0$. 
\begin{enumerate}[(i)]
\item For $\tau \in (0,\widetilde{\tau}_1)$, there exists a unique solution to $F(\tau,y)=0$, which is $C^1$ and can be parametrized as $y_1 (\tau)$.
\item At each $\widetilde{\tau}_j$, $j\in\N$, two new solutions to $F(\tau, y)=0$ are born, both of which can be parametrized in terms of $\tau$. We denote these solutions by $y_{2j}(\tau)$ and $y_{2j+1}(\tau)$ for $\tau\geq \widetilde{\tau}_j$. They satisfy $y_{2j}(\widetilde{\tau}_j)  =y_{2j+1}(\widetilde{\tau}_j)=\lambda$ and
\begin{equation}\label{eq:birth_derivative}
\lim_{\tau\rightarrow \widetilde{\tau}_j^{+}} y_{2j}'(\tau)  = \frac{-2\,\lambda^3}{\pi j-1}<0, \qquad
\lim_{\tau\rightarrow \widetilde{\tau}_j^{+}} y_{2j+1}'(\tau)  = \frac{-2\,\lambda^3}{\pi j+1}<0.
\end{equation}
\item There is no other time at which solutions are born except $\widetilde{\tau}_j$, $j\in\N$.
\item $y_{2j}(\tau)$ is decreasing for all times of existence and $y_{2j}(\tau)<y_{2j+1}(\tau)$ for all $\tau>\widetilde{\tau}_j$ where both solutions exist.
\item For $\tau \in (\widetilde{\tau}_{j-1}, \tau_j)$, $j\in\N$, the solution $y_{2j-1}(\tau)$ is decreasing. At $\tau_j$, $y_{2j-1}(\tau)$ reaches $\lambda/\sqrt{2}$ and starts to increase.
\item For each $j\in\N$, there exists a time $\tau_j^{\infty} \in (\tau_j, \tau_{j+1})$ where $y_{2j-1}$ collides with $y_{2j}$, i.e. $y_{2j-1}(\tau_j^{\infty})=y_{2j}(\tau_j^{\infty})$, $y'_{2j-1}(\tau_j^{\infty})=+\infty$, $y'_{2j}(\tau_j^{\infty})=-\infty$, and these solutions cease to exist after $\tau_j^{\infty}$.
\item Solutions do not intersect unless they collide.
\item Solutions do not cease to exist unless they collide.
\end{enumerate}
\end{prop}

\begin{center}
\begin{figure}[htp]
\begin{tikzpicture}
\draw[thick,->] (0,0) -- (12.5,0) node[anchor=north west] {$\tau$};
\draw[thick,->] (0,0) -- (0,6.5) node[anchor=south east] {$y(\tau)$};
\draw (2 cm,1pt) -- (2 cm,-1pt) node[anchor=north] {$\substack{\widetilde{\tau}_1\\ \tau_1}$};
\draw (4 cm,1pt) -- (4 cm,-1pt) node[anchor=north] {$\substack{\widetilde{\tau}_2\\ \ }$};
\draw (6 cm,1pt) -- (6 cm,-1pt) node[anchor=north] {$\substack{\widetilde{\tau}_3\\ \tau_2}$};
\draw (8 cm,1pt) -- (8 cm,-1pt) node[anchor=north] {$\substack{\widetilde{\tau}_4\\ \ }$};
\draw (10 cm,1pt) -- (10 cm,-1pt) node[anchor=north] {$\substack{\widetilde{\tau}_5\\ \tau_3}$};
\draw (1pt,2.5 cm) -- (-1pt,2.5 cm) node[anchor=east] {$\frac{\lambda}{\sqrt{2}}$};
\draw (1pt,5 cm) -- (-1pt,5 cm) node[anchor=east] {$\lambda$};

\draw[gray,thick,dashed] (0,2.5) -- (12,2.5);
\draw[gray,thick,dashed] (0,5) -- (12,5);

\draw[red,thick] (0,5) .. controls (1,4.5) and (1,2.5) .. (2,2.5) ;
\draw [red,thick] plot [smooth, tension=1.3] coordinates { (2,2.5) (3,2.8) (3.5,3.5)};
\draw [blue,thick] plot [smooth, tension=1.3] coordinates { (2,5) (3,4.5) (3.5,3.5)};

\draw[red,thick] (2,5) .. controls (4,5) and (4,2.5) .. (6,2.5) ;
\draw[red,thick] (6,2.5) arc (270:350:0.82cm);

\draw [blue,thick] plot [smooth, tension=1.3] coordinates { (4,5) (6,4.1) (6.78,3.15)};

\draw[red,thick] (4,5) .. controls (7,5) and (7,2.5) .. (10,2.5) ;
\draw[red,thick] (10,2.5) arc (270:350:0.3 cm);

\draw [blue,thick] plot [smooth, tension=1.1] coordinates { (6,5) (9.1,4.3) (10.3,2.74)};

\draw[red,thick] (6,5) .. controls (10,5) and (10,3.5) .. (12,2.9) ;
\draw [red,thick] plot [smooth, tension=1.1] coordinates { (8,5) (10,4.92) (12,4.6)};
\draw [red,thick] plot [smooth, tension=1.1] coordinates { (10,5) (11,4.97) (12,4.9)};

\draw [blue,thick] plot [smooth, tension=1.1] coordinates { (8,5) (10,4.8) (12,4.2)};
\draw [blue,thick] plot [smooth, tension=1.1] coordinates { (10,5) (11,4.96) (12,4.85)};

\node at (3.5,3.5) [circle,fill=black,inner sep=1pt]{};
\node at (6.81,3.17) [circle,fill=black,inner sep=1pt]{};
\node at (10.3,2.74) [circle,fill=black,inner sep=1pt]{};

\node at (2,2.5) [circle,fill=red,inner sep=1pt]{};
\node at (6,2.5) [circle,fill=red,inner sep=1pt]{};
\node at (10,2.5) [circle,fill=red,inner sep=1pt]{};

\node at (2,5) [circle,fill=red!50!blue!90,inner sep=1pt]{};
\node at (4,5) [circle,fill=red!50!blue!90,inner sep=1pt]{};
\node at (6,5) [circle,fill=red!50!blue!90,inner sep=1pt]{};
\node at (8,5) [circle,fill=red!50!blue!90,inner sep=1pt]{};
\node at (10,5) [circle,fill=red!50!blue!90,inner sep=1pt]{};

\node at (1.5,3.5) [red] {$y_1$};
\node at (3,4) [blue] {$y_2$};
\node at (4.7,3.5) [red] {$y_3$};
\node at (5.5,4) [blue] {$y_4$};
\node at (8,3.5) [red] {$y_5$};
\node at (9.1,4) [blue] {$y_6$};
\node at (11.3,3.5) [red] {$y_7$};
\node at (11.9,4) [blue] {$y_8$};
\end{tikzpicture}
\caption{{\small Schematic drawing of solutions $(\tau,y(\tau))$ to $F(\tau,y)=0$.}}
\label{fig:solution_y}
\end{figure}
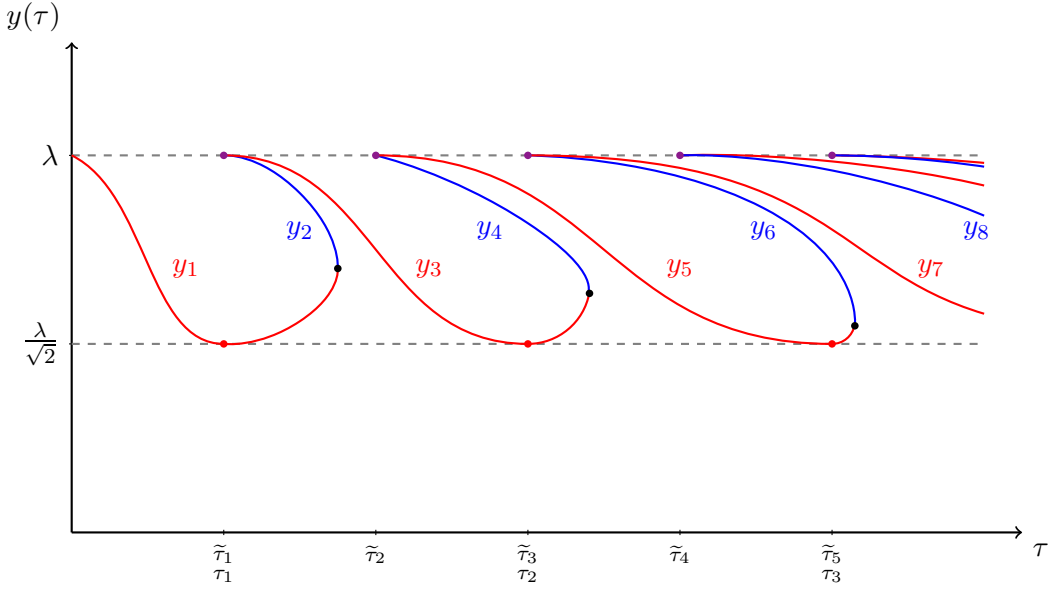
\end{center}

\begin{proof}
{\bf \emph{(ii)}} Let $j\in\N$. The point $(\widetilde{\tau}_j, \lambda)$ belongs to $F(\tau,y)=0$ by direct inspection. At this point, $h$ is not differentiable and therefore neither is $F$. However, \Cref{thm:curve_tau_xi} implies that a Lipschitz solution curve exists traversing this point. Consider a sufficiently small neighborhood of $(\widetilde{\tau}_j, \lambda)$ in our curve. The function $F$ is differentiable at any point of this neighborhood except for $(\widetilde{\tau}_j, \lambda)$. 

In the $(\tau,\xi)$ coordinates introduced in \Cref{thm:curve_tau_xi}, this point can be written as $(\widetilde{\tau}_j, \frac{\pi\, j}{2})$. As shown in \Cref{thm:curve_tau_xi}, a global parametrization of our curve in these new coordinates is
 $\tau (\xi)= \xi \, h(\xi)^2/(2\lambda^2)$. A neighborhood of $(\widetilde{\tau}_j, \frac{\pi\, j}{2})$ in our curve is given by 
 \[
 \{ (\tau (\xi), \xi) \mid \xi \in (\frac{\pi\, j}{2}-\kappa, \frac{\pi\, j}{2}+\kappa)\}
 \]
 for some $\kappa>0$ sufficiently small, which we split into two parts:
  \[
\mathcal{C}_{-} = \{ (\tau (\xi), \xi) \mid \xi \in (\frac{\pi\, j}{2}-\kappa, \frac{\pi\, j}{2})\}, \qquad \mathcal{C}_{+} = \{ (\tau (\xi), \xi) \mid \xi \in (\frac{\pi\, j}{2}, \frac{\pi\, j}{2}+\kappa)\}.
 \]
Let $\widetilde{\Cc}_{\pm}$ be the image of these neighborhoods in $(\tau,y)$ coordinates. Along these curves, we may take the one-sided limits  $\xi\rightarrow \pi j/2$ of the function $h(\xi) + 2\xi h'(\xi)$, which are $1 \pm \pi j$. 

Since the limits are nonzero,  by restricting $\kappa$, we may show that $\pa F/\pa y$ is nonzero on each $\widetilde{\Cc}_{\pm}$. In particular, at every point in $\widetilde{\Cc}_{-}$ (resp. $\widetilde{\Cc}_{+}$), our curve is differentiable and it may be parametrized in terms of $(\tau,y(\tau))$ coordinates (resp. $(\tau,\tilde{y}(\tau))$. Moreover, by further restricting $\kappa$, we may show that $\tau(\xi)$ is decreasing for $\xi$ in $\Pi_2 \mathcal{C}_{-}$ and increasing for $\xi$ in $\Pi_2 \mathcal{C}_{+}$, i.e. $\pi j/2$ is a local minimum of the curve $\tau (\xi)$. As a result, $\Pi_1 (\widetilde{\Cc}_{+}\cup \widetilde{\Cc}_{-})\subset (\widetilde{\tau}_j,\infty)$, where we denote by $\Pi_j$ the projection onto the $j$-th coordinate.
 
We may compute the one-sided limits
 \begin{equation}\label{eq:birth_sol_der}
 \begin{split}
 y'(\widetilde{\tau}_j) & = \lim_{\tau\rightarrow \widetilde{\tau}_j^{+}} \frac{-2\, y(\tau)^3\, h'(2\tau y(\tau)^2)}{h(2\tau y(\tau)^2) + 4\tau y(\tau)^2\, h'(2\tau y(\tau)^2)} = \frac{2\, \lambda^3}{1 - \pi\, j}<0\\
  \tilde{y}'(\widetilde{\tau}_j) & = \lim_{\tau\rightarrow \widetilde{\tau}_j^{+}} \frac{-2\, \tilde{y}(\tau)^3\, h'(2\tau \tilde{y}(\tau)^2)}{h(2\tau \tilde{y}(\tau)^2) + 4\tau \tilde{y}(\tau)^2\, h'(2\tau \tilde{y}(\tau)^2)} = \frac{-2\, \lambda^3}{1 + \pi\, j}<0.
 \end{split}
 \end{equation}
 
As a result, $y$ and $\tilde{y}$ must be two different solutions (since their one-sided derivatives at those points are different) and they both exist for times $\tau>\widetilde{\tau}_j$. Moreover, neither solution could exist before $\widetilde{\tau}_j$, as this would imply that the curve $F=0$ intersects itself at the point $(\widetilde{\tau}_j,\lambda)$. But this would imply that $\widetilde{F}=0$ passes through $(\widetilde{\tau}_j,\pi j/2)$ twice, which is false.

From now on, we let $y_{2j}$ and $y_{2j+1}$ be the two solutions born at $\widetilde{\tau}_j$.

{\bf \emph{(iii)}} Suppose a solution $y$ is born at a certain time $0<\tau_y\neq \widetilde{\tau}_j$ for all $j\in\N$. The point $(\tau_y, y(\tau_y))=(\tau_y, y_0)$ is in the Lipschitz-continuous curve $F(\tau,y)=0$. Since $\tau_y\neq \widetilde{\tau}_j$, our curve must be differentiable at this point. There are two options: either $(\pa_y F)(\tau_y,y_0)=0$ or it is nonzero. In the latter case, the Implicit  Function theorem implies that $y$ can be solved as a function of $\tau$ in a neighborhood of $\tau_y$, which effectively extends the solution for times $\tau<\tau_y$, a contradiction. 

Suppose that 
\begin{equation}\label{eq:yderiv_F}
(\pa_y F)(\tau_y,y_0)= h(2\tau_y y_0^2) + 4\tau_y y_0^2\, h'(2\tau_y y_0^2)= 0
\end{equation}
Then one may easily check that $(\pa_{\tau} F)(\tau_y,y_0)\neq 0$, and the curve $F(\tau, y)=0$ can be extended around the point $(\tau_y,y_0)$ for $\tau$ as a function in $y$.

In order to study this function, let us parametrize the curve in $(\tau(\xi),\xi)$ coordinates as in \Cref{thm:curve_tau_xi}. For $\xi_0 := 2 \tau_y  y_0^2$, we have that $\tau'(\xi_0)=0$. The fact that $h''(\xi)=-h(\xi)$ in all points of differentiability, together with \eqref{eq:yderiv_F}, yields
\[
\begin{split}
\tau''(\xi) \Big |_{\xi=\xi_0} & =   \frac{4h(\xi) h'(\xi)+2\xi h'(\xi)^2+2 \xi h(\xi)\, h''(\xi) }{2\lambda^2}\Big |_{\xi=\xi_0} =
\frac{- \frac{2h(\xi_0)^2}{\xi_0}+\frac{h(\xi_0)^2}{2\xi_0}- 2 \xi_0 h(\xi_0)^2}{2\lambda^2}\\
&= \frac{h(\xi_0)^2}{4\lambda^2\,\xi_0} \, ( -3-4\xi_0^2) <0 .
\end{split}
\]
Back to the parametrization $(\tau (y), y)$ near $(\tau_y, y_0)$, the chain rule yields
\[
\tau''(y) \Big |_{y=y_0} = \frac{d^2\tau}{d\xi^2}\Big |_{\xi_0}\, \left( \frac{d\xi}{dy}\Big |_{y_0}\right)^2<0.
\]
As a result, we have that the curve $F(\tau,y)=0$ is left of the vertical asymptote at $(\tau_y,y_0)$, and thus the solution $y$ born at $\tau_y$ cannot be continued for $\tau>\tau_y$. This is contradiction, which proves that no solutions can be born at $\tau_y$.

{\bf \emph{(i)}} The existence of at least one solution curve starting at $(\tau,y)=(0,\widetilde{z}_0)$ was proved in \Cref{thm:existence}; we call this solution $y_1$. In order to show that $y_1$ is the only solution born at $\tau=0$, let us consider the curve $F(\tau,y)=0$ for $\tau<0$. By definition of $F$, the curve must be even with respect to $\tau$. As a result, one may extend $y_1$ to negative values of $\tau$ by reflection (note that $\lim_{\tau\rightarrow 0^{+}} y_1'(\tau)<0$).
By the local uniqueness of solutions (which follows from Implicit Function theorem for Lipschitz functions \cite{Clarke1,Clarke2}), $y_1$ must be the only solution that traverses $\tau=0$.

Finally, \emph{(iii)} implies that no other solution is born in the time interval $(0,\widetilde{\tau}_1)$. Since at least one solution must exist (by \Cref{thm:existence}), $y_1$ must exist for all $\tau \in [0,\widetilde{\tau}_1)$.

{\bf \emph{(viii)}} Suppose that a solution $y(\tau)$ exists for $\tau \in [a,b)$, $0<b<\infty$, its maximal interval of existence.
We want to prove that $\lim_{\tau\rightarrow b^{-}} y'(\tau)$ is $\pm \infty$.

By contradiction, suppose that is not the case (thus this limit either does not exist or it is finite). For each $\tau\in (a,b)$, $\tau\neq \widetilde{\tau}_j$, the function $h(2\tau y(\tau)^2)$ is differentiable. Consider $\kappa>0$ sufficiently small so that $(b-\kappa,b)\cap \{\widetilde{\tau}_j\}_{j\in\N}=\emptyset$. By the Implicit Function theorem, $y(\tau)$ is differentiable and
\begin{equation}\label{eq:y_prime}
y'(\tau)= \frac{-2\, y(\tau)^3\, h'(2\tau y(\tau)^2)}{h(2\tau y(\tau)^2) + 4\tau y(\tau)^2\, h'(2\tau y(\tau)^2)}
\end{equation}
provided the denominator does not vanish. In particular, note that the numerator and denominator cannot simultaneously vanish. Since $y'(b)\neq \pm\infty$, the denominator cannot possibly vanish at $b$, and therefore $y(\tau)$ must be differentiable in a small neighborhood $(b-\kappa,b)$, by making $\kappa>0$ sufficiently small.

In this neighborhood,
\begin{equation}\label{eq:denom}
h(2\tau y(\tau)^2) + 4\tau y(\tau)^2\, h'(2\tau y(\tau)^2)\neq 0
\end{equation}
and therefore the denominator must have a constant sign for all $\tau \in (b-\kappa,b)$.

A simple computation then shows that for all $\tau \in (b-\kappa,b)$, 
\begin{equation}\label{eq:xi_prime}
[2\tau \, y(\tau)^2]' = \frac{2\lambda\, y(\tau)}{h(2\tau y(\tau)^2) + 4\tau y(\tau)^2\, h'(2\tau y(\tau)^2)}.
\end{equation}
The fact that the denominator has a sign implies that $2\tau \, y(\tau)^2$ is either increasing or decreasing in the interval $(b-\kappa,b)$. In particular, the limit $\lim_{\tau\rightarrow b^{-}} 2\tau y(\tau)^2=:\xi_b$ must exist since the denominator does not blow up. Moreover $\xi_b\neq 0,\infty$ since $b\neq 0,\infty$ and $y(\tau)\in [\lambda/\sqrt{2}, \lambda]$. As a result, $y(\tau)$ converges to $\sqrt{\xi_b/(2b)}$ as $\tau\rightarrow b^{-}$.

Now there are two possibilities: either $h$ is differentiable at $\xi_b$ or it is not. If it isn't, we have that $\xi_b$ must be of the form $\pi j/2$ and $b=\widetilde{\tau}_j$ for some $j\in\N$. By \emph{(ii)}, $y$ must be one of the solutions that are born at that point, but this contradicts the fact that $y$ existed for times $\tau<b$.
If $h$ is differentiable at $\xi_b$, then we may take the limit $\tau\rightarrow b^{-}$ at \eqref{eq:denom}, which implies that the denominator does not vanish at $\tau=b$ (by hypothesis). By the Implicit Function theorem, the curve $F(\tau,y)=0$ can be parametrized as $(\tau,y(\tau))$ on a neighborhood of the point $(b, \sqrt{\xi_b/(2b)})$, where $y(\tau)$ is locally unique and Lipschitz. This effectively extends the solution past $\tau=b$ which is a contradiction.

We have therefore proved that when a solution ceases to exists, we must have $\lim_{\tau\rightarrow b^{-}} y'(\tau)\in \{ \pm \infty \}$. Suppose that this limit is $+\infty$ (the opposite case is analogous). We must show that there exists another solution $\tilde{y}$ that collides with $y$, i.e. 
\begin{equation}
\lim_{\tau\rightarrow b^{-}} \tilde{y}(\tau)=y(b) ,\qquad \mbox{and} \qquad \lim_{\tau\rightarrow b^{-}} \tilde{y}'(\tau)=-\infty.
\end{equation}

Note that $y'(\tau)$ does not blow up at $\widetilde{\tau}_j$, and thus we can find $\kappa>0$ small enough so that $(b-\kappa,b)$ does not intersect either point (thus $h$ is differentiable in this interval). Since $y'(b)=+\infty$, we have that the denominator \eqref{eq:denom} vanishes at $b$.
 
In view of \eqref{eq:denom},  this implies that $h'(2b y(b)^2)<0$. By continuity, and by making $\kappa$ smaller if necessary, we may assume that $h'(2\tau y(\tau)^2)<0$ for all $\tau \in (b-\kappa,b)$. In view of \eqref{eq:y_prime}, $y'(\tau)$ has a sign in $(b-\kappa,b)$, and since $y'(b)=+\infty$, it must be positive. Thus $y(\tau)$ is increasing for $\tau \in (b-\kappa,b)$.

Using the fact that \eqref{eq:denom} vanishes at $b$, it follows that 
\[
\frac{\pa F}{\pa \tau} (b,y(b)) = 2\, y(\tau)^3 \, h'(2\tau y(\tau)^2) \Big |_{\tau=b} = -\frac{2\,y(b)^3\, h(2b y(b)^2)}{4b y(b)^2} =-\frac{\lambda}{2b}  \neq 0.
\]
By the Implicit Function theorem, the curve $F(\tau,y)=0$ admits a parametrization of the form $(\tau(y),y)$ around the point $(b,y(b))$. By a similar argument to \emph{(iii)}, one easily sees that this curve can be extended beyond this point for values $y>y(b)$. This necessarily corresponds to a different solution $\tilde{y}(\tau)$ (since our solution $y(\tau)$ was strictly smaller than $y(b)$ for all $\tau\in (b-\kappa,b)$). 

This proves that $\tilde{y}(b)=y(b)$. Solution $\tilde{y}(\tau)$ cannot exist beyond $\tau=b$, and thus we have proved that $\tilde{y}'(b)\in \{\pm \infty\}$. The fact that the sign is opposite to that of $y'(b)$ follows from the next point.

{\bf \emph{(vii)}}  Suppose that two solutions $y(\tau)$ and $\tilde{y}(\tau)$, which exist in a common time interval $\tau\in(b-\kappa,b)$ intersect at $\tau=b$ ($0<\kappa<b<\infty$) for the first time. 
We may assume that $h'$ (and thus $F$) is differentiable in this interval by reducing $\kappa$ if necessary. By \emph{(viii)}, $y'(\tau)$ and $\tilde{y}'(\tau)$ must exist in this interval, otherwise the solution would cease to exist.
By continuity, we may assume that the denominator in $y'$ and $\tilde{y}'$, \eqref{eq:denom}, has a sign for all $\tau\in (b-\kappa,b)$. By \eqref{eq:xi_prime}, $[2\tau y(\tau)^2]'$ (resp.\ $[2\tau \tilde{y}(\tau)^2]'$) also has a sign. This implies that $2\tau y(\tau)^2$ (resp. $2\tau \tilde{y}(\tau)^2$) is monotone and we may therefore assume $h'(2\tau y(\tau)^2)$ has a constant sign in $(b-\kappa,b)$ by further reducing $\kappa$. The latter, together with the sign of the denominator, implies that $y'(\tau)$ and $\tilde{y}'(\tau)$ have a sign in $(b-\kappa,b)$. Thus $y$ and $\tilde{y}$ are monotone in this interval.

By implicit differentiation, 
\begin{equation}\label{eq:intersection}
\begin{split}
\frac{\pa F}{\pa \tau} (\tau, y(\tau)) + \frac{\pa F}{\pa y} (\tau, y(\tau)) \cdot y'(\tau) & = 0\\
\frac{\pa F}{\pa \tau} (\tau, \tilde{y}(\tau)) + \frac{\pa F}{\pa y} (\tau, \tilde{y}(\tau)) \cdot \tilde{y}'(\tau) & = 0.
\end{split}
\end{equation}
Subtracting these equations, and taking the limit as $\tau\rightarrow b^{-}$ we find that 
\[
\frac{\pa F}{\pa y} (b, y(b)) \cdot \left[ y'(b) - \tilde{y}'(b)\right] = 0.
\]
Thus either $\lim_{\tau\rightarrow b^{-}} [y'(\tau) - \tilde{y}'(\tau)]=0$, or $\frac{\pa F}{\pa y} (b, y(b))=0$. 

\indent \emph{Case 1)} First suppose that $\frac{\pa F}{\pa y} (b, y(b))=0$. If $y'(b)$ is finite, then it quickly follows from either equation in \eqref{eq:intersection} that $\frac{\pa F}{\pa \tau} (b, y(b))=0$. However, one may easily check that along our curve $F=0$ we have
\begin{equation}\label{eq:Fdery_dertau}
\frac{\pa F}{\pa y} (\tau,y) = \frac{2\tau}{y}\, \frac{\pa F}{\pa \tau} (\tau,y) + \frac{\lambda}{y},
\end{equation}
and therefore both partial derivatives can't simultaneously vanish at $(b,y(b))$. It follows that both $y'(b)$ and $\tilde{y}'(b)$ must blow up. Moreover, the curve $F(\tau,y)=0$ is smooth near $(b,y(b))$ (since $y'$ does not blow up at $\tau=\widetilde{\tau}_j$ and thus $b\neq \widetilde{\tau}_j$). Near this point, $\tau$ can be written as a $C^1$ function of $y$ and by implicit differentiation
\[
\frac{\pa F}{\pa \tau} (b,y(b)) \, \tau'(y(b)) =0.
\]
By \eqref{eq:Fdery_dertau}, we must have that $\tau'(y(b))=0$ and thus $F(\tau,y)=0$ has a vertical tangent at $(b,y(b))$ (see also the argument in \emph{(iii)}). As a result $\tilde{y}'$ and $\tilde{y}$ must have opposite signs, and without loss of generality $y'(b)=+\infty$ while $\tilde{y}'(b)=-\infty$.

\indent \emph{Case 2)} Next suppose that $\lim_{\tau\rightarrow b^{-}} [y'(\tau) - \tilde{y}'(\tau)]=0$. This, together with the fact that $y$ and $\tilde{y}$ are monotone in $(b-\kappa,b)$, implies that $y'(\tau)$ and $\tilde{y}'(\tau)$ must have the same sign and that 
\[
\lim_{\tau\rightarrow b^{-}} y'(\tau)
\]
either is finite or it is $\pm \infty$. We rule out the case $\pm\infty$ since that would imply that the denominator in \eqref{eq:y_prime} must vanish, and thus $\frac{\pa F}{\pa y} (b, y(b))=0$ which is Case 1. Let us therefore assume that
\[
\lim_{\tau\rightarrow b^{-}} y'(\tau) = m_b = \lim_{\tau\rightarrow b^{-}} \tilde{y}'(\tau) ,
\]
for some $m_b\in\R$. Then either $b=\widetilde{\tau}_j$ for some $j\in\N$, or the curve $F(\tau,y)=0$ must be $C^{1}$ at $(b,y(b))$.

\begin{itemize}
\item  If $b=\widetilde{\tau}_j$, then $(b,y(b))=(\widetilde{\tau}_j, \widetilde{z}_0)$. By \emph{(ii)}, $y$ and $\tilde{y}$ must be $y_{2j}$ and $y_{2j+1}$. However, neither of these solutions could have existed for $\tau< \widetilde{\tau}_j$, a contradiction.
\item Suppose that the curve $F(\tau,y)=0$ is $C^1$ at $(b,y(b))$. Since $\frac{\pa F}{\pa y}(b,y(b))\neq 0$, the Implicit Function theorem allows us to extend $y$ and $\tilde{y}$ past the point $\tau=b$. This contradicts the local uniqueness of the curve $F(\tau,y)=0$ at $(b,y(b))$.
\end{itemize}

We have therefore proved that Case 1 is the only option when two solutions $y$ and $\tilde{y}$ intersect, namely they must collide, $y'(b)=-\tilde{y}'(b)=+\infty$ and both solutions cease to exist beyond $\tau=b$.

{\bf \emph{(iv)}} By point \emph{(vii)}, the two solutions $y_{2j}(\tau)$ and $y_{2j+1}(\tau)$ may not intersect unless they collide. For $\tau>\widetilde{\tau}_j$ sufficiently small, $y_{2j}(\tau)<y_{2j+1}(\tau)$ in view of their derivatives at $\widetilde{\tau}_j$, cf. \eqref{eq:birth_sol_der}. For these solutions to collide $y_{2j}'(\tau)$ would have to change sign. Note that  
\begin{equation}\label{eq:xi_der_2j}
(2\tau y_{2j}(\tau)^2)' = \frac{2\lambda\, y_{2j}(\tau)}{h(2\tau y_{2j}(\tau)^2) + 4\tau y_{2j}(\tau)^2\, h'(2\tau y_{2j}(\tau)^2)}.
\end{equation}
The denominator is negative at $\tau=\widetilde{\tau}_j$ and may not change unless $y_{2j}'(\tau)$ blows up. As a result, $2\tau y_{2j}(\tau)^2$ is decreasing for all times of existence of $y_{2j}$. In particular, $h'(2\tau y_{2j}(\tau))<0$ until $\tau=\tau_k>\widetilde{\tau}_j$ for some $k\in\N$.

Consequently, the derivative
\[
y_{2j}'(\tau)= \frac{-2y_{2j}(\tau)^3 h'(2\tau y_{2j}(\tau)^2)}{h(2\tau y_{2j}(\tau)^2) + 4\tau y_{2j}(\tau)^2\, h'(2\tau y_{2j}(\tau)^2)}.
\]
is negative until $2\tau y_{2j}(\tau)^2=\frac{\pi}{2} (k-1/2)$ for some $k\in\N$, i.e. $\tau=\tau_k$ and $y(\tau)=\lambda/\sqrt{2}$. At $\tau=\tau_k$, the denominator becomes is strictly positive and equal to $\sqrt{2}$. By continuity, this implies that the denominator must vanish before $y_{2j}'(\tau)$ changes sign, and thus $y_{2j}(\tau)$ must cease to exist before it can collide with $y_{2j+1}(\tau)$.

Consequently, $y_{2j}$ must collide with some solution $y_{i}$ for $i<2j$, and therefore $y_{2j}(\tau)<y_{2j+1}(\tau)$ for all common times of existence $\tau$.

{\bf \emph{(v)} and \emph{(vi)}} We argue by induction. Consider the case $j=1$. The solution $y_1(\tau)$ is decreasing for $\tau<\tau_1=\widetilde{\tau}_1$ by direct inspection of the derivative \eqref{eq:y_prime} and \eqref{eq:xi_prime}. At $\tau=\tau_1$, $y_1(\tau_1)=\lambda/\sqrt{2}$. 

The denominator of the derivative \eqref{eq:denom} is strictly positive at $\tau_1$. However, $h'(2\tau\, y_1(\tau)^2)<0$ for all $\tau>\tau_1$. As a result, $y_1$ must cease to exist before it reaches $\lambda$. By \emph{(viii)}, $y_1$ must collide with another solution at a certain time $\tau_1^{\infty}$.  For times $\tau>\tau_1$, there exist at least two more solutions $y_2(\tau)<y_3 (\tau)$ by \emph{(ii)} and \emph{(iv)}. At time $\tau_2$ another solution must reach $\lambda/\sqrt{2}$. This solution cannot have collided with $y_1$ or it would have ceased to exist. As a result, $\tau_1^{\infty}<\tau_2=\widetilde{\tau}_3$. In particular, $y_1$ must collide with $y_2$, $y_3$, $y_4$ or $y_5$. However $y_2<y_3<y_4<y_5$ since they are born ordered and, by \emph{(vii)}, they can't intersect unless they collide (but they can't collide before they become increasing after hitting $\lambda/\sqrt{2}$, which is impossible before one collides with $y_1$). As a result, $y_1$ must collide with $y_2$. This also implies that $y_3$ must reach $\lambda$ at time $\tau_2$. 

We have proved \emph{(v)} and \emph{(vi)} in the case $j=1$. Next suppose that these claims are true for all integers smaller than $j\geq 2$, and let us prove it for $j$. Consider the solutions $y_{2j-1}$ and $y_{2j}$ born at times $\widetilde{\tau}_{j-1}$ and $\widetilde{\tau}_{j}$, respectively. By \emph{(ii)}, $y_{2j-1}$ has a negative derivative at birth. For $\tau>\widetilde{\tau}_{j-1}$, we have that 
\[
y'_{2j-1}(\tau) = \frac{-2\, y_{2j-1}(\tau)^3\, h'(2\tau y_{2j-1}(\tau)^2)}{h(2\tau y_{2j-1}(\tau)^2) + 4\tau y_{2j-1}(\tau)^2\, h'(2\tau y_{2j-1}(\tau)^2)} .
\]
At time $\widetilde{\tau}_{j-1}$ the (right-limit of the) denominator is positive. As a result,
\[
(2\tau y_{2j-1}(\tau)^2)' = \frac{2\lambda y_{2j-1}(\tau)}{h(2\tau y_{2j-1}(\tau)^2) + 4\tau y_{2j-1}(\tau)^2\, h'(2\tau y_{2j-1}(\tau)^2)}
\]
is positive at $\widetilde{\tau}_{j-1}$. By inspection of this denominator we notice that it stays positive as $2\tau y_{2j-1}(\tau)^2$ goes from $2\widetilde{\tau}_{j-1} \lambda^2= \frac{\pi}{2}\, (j-1)$ until $\frac{\pi}{2}\, (j-1) + \frac{\pi}{4}$ (and thus the solution cannot collide in this time). At the latter point, $\tau=\tau_j$ and $y_{2j-1}(\tau_j)=\lambda/\sqrt{2}$. This proves \emph{(v)} for $y_{2j-1}$. A similar argument shows that $y_{2j}$ is decreasing until it ceases to exist.

For $\tau>\tau_j$, $y_{2j-1}'$ becomes positive. By the induction hypothesis, there exists a time $\tau_{j-1}^{\infty} <\tau_j$ where $y_{2j-3}$ collided with $y_{2j-2}$. As a result, $y_{2j-1}$ is the smallest solution existing at time $\tau_j$. 
At $\tau=\tau_j$, the denominator of $y_{2j-1}'$ equals $1$. For times $\tau>\tau_j$ and until this denominator becomes zero, $(2\tau y_{2j-1}(\tau)^2)'>0$. If $2\tau y_{2j-1}(\tau)^2$ reached $\frac{\pi}{2} \, j$, then the denominator would become $1- \pi \, j<0$. This cannot happen without the denominator passing through zero, and thus $y_{2j-1}$ ceasing to exist. As a result, there exists some time $\tau_j^{\infty} >\tau_j$ such that $y_{2j-1}$ ceases to exist. By \emph{(viii)}, this solution must collide with another. Moreover $\tau_j^{\infty}<\tau_{j+1}$ since at $\tau_{j+1}$ some decreasing solution must reach $\lambda/\sqrt{2}$ which is impossible without intersecting (and thus colliding with) $y_{2j-1}$ first. As a result, $y_{2j-1}$ cannot exist at time $\tau_{j+1}$. Thus $\tau_j^{\infty}\in (\tau_j,\tau_{j+1})$.

Finally, for times $\tau\in (\tau_j,\tau_{j+1})$, only solutions $y_{2j},y_{2j+1}, \ldots, y_{4j+1}$ may exist. These solutions are ordered increasingly by \emph{(iv)} and the fact that they are all decreasing until they reach $\lambda/\sqrt{2}$ (which may not happen until $\tau_{j+1}$). As a result, $y_{2j-1}$ must collide with $y_{2j}$. This concludes the proof of claims \emph{(v)} and \emph{(vi)}.
\end{proof}

As a result of \Cref{thm:bifurcations}, there are as many $\tau_j^{\infty}$ as $\tau_j$, which live between $\tau_j$ and $\tau_{j+1}$. We will later show that $\tau_j^{\infty}$ approaches $\tau_j$ as $j\rightarrow\infty$. However, the fact that the $\tau_j$ are equally spaced allows us to compute the approximate rate between births and deaths of solutions.

At each birth time $\widetilde{\tau}_j$, two solutions are born. At each collision time $\tau_j^{\infty}$, two solutions die.
The  approximate rate between births and deaths is thus
\[
\lim_{j\rightarrow\infty} \frac{\tau_j}{\widetilde{\tau}_j} = 2
\]
Consequently, as time passes, more and more solutions exist simultaneously. Another consequence of \Cref{thm:bifurcations} is that simultaneous solutions are truly ordered according to their subindex.

\subsection{Minimal solution and asymptotic behavior}

Our goal is to identify and study the smallest solution to \eqref{eq:def_h}, which we call $Y(\tau)$. \Cref{thm:bifurcations} shows that 
\begin{equation}\label{eq:big_Y}
Y (\tau)= \begin{cases}
y_1(\tau) & \mbox{if}\ \tau \in (0,\tau_1^{\infty}],\\
y_{2j+1} (\tau) & \mbox{if}\ \tau \in (\tau_j^{\infty}, \tau_{j+1}^{\infty}], \quad j\in\N.
\end{cases}
\end{equation}
This function of $\tau$ is left-continuous and it has jumps at the points $\tau_j^{\infty}$.  Moreover, we define the function 
\begin{equation}\label{eq:def_J}
\Jc (\lambda,\tau)= \lim_{\delta\rightarrow 0^{+}} Y(\lambda, \tau +\delta) 
=\begin{cases}
y_1(\tau) & \mbox{if}\ \tau \in (0,\tau_1^{\infty}),\\
y_{2j+1} (\tau) & \mbox{if}\ \tau \in [\tau_j^{\infty}, \tau_{j+1}^{\infty}), \quad j\in\N,
\end{cases}
\end{equation}
which appears in the lower bound in \Cref{thm:resonantLDP_diff_variance}.
Before we continue our study of the function $Y$, we obtain a good approximation of the collision times $\tau_j^{\infty}$.

\begin{prop}\label{thm:jump_points} Fix $\lambda>0$ and let $\tau_j$, $j\in\N$, be as in \eqref{eq:def_tau}. As $j\rightarrow\infty$, we have the following asymptotic development for the collision times
\[
\tau_j^{\infty}= \tau_j + \O\left( \frac{1}{j\, \lambda^2}\right).
\]
\end{prop}
\begin{proof}
By \Cref{thm:bifurcations}, we know that $y_{2j-1}$ and $y_{2j}$ collide at time $\tau_j^{\infty}\in (\tau_j,\tau_{j+1})$. 
Let 
\[
\xi_{2j-1}(\tau):= 2\tau y_{2j-1}(\tau)^2,\quad \xi_{2j}(\tau):= 2\tau y_{2j}(\tau)^2.
\]

By \eqref{eq:xi_der_2j}, $\xi_{2j}'(\tau)$ is negative at birth ($\tau=\widetilde{\tau}_j$) and must stay so until collision. Since $\widetilde{\tau}_j<\tau_j<\tau$,
\[
\xi_{2j}(\tau) < \xi_{2j}(\tau_j) < \xi_{2j}(\widetilde{\tau}_j )  = \frac{\pi}{2}\, j.
\]
An analogous argument with $y_{2j-1}$ shows that for $\tau>\tau_j$
\[
\xi_{2j-1}(\tau)> \xi_{2j-1}(\tau_j) = \frac{\pi}{2}\, \left( j -\frac{1}{2}\right).
\]
The fact that $y_{2j-1}(\tau)\leq y_{2j}(\tau)$ yields
\begin{equation}\label{eq:first_bound}
\frac{\pi}{2}\, \left( j -\frac{1}{2}\right) < \xi_{2j-1}(\tau_j^{\infty} ) <  \frac{\pi}{2}\, j.
\end{equation}
In particular, $h$ is differentiable at all such values.

On the other hand, the denominator \eqref{eq:denom} of $y_{2j-1}'$ must vanish at collision time, which yields the equality:
\[
2\,  \xi_{2j-1}(\tau_j^{\infty} ) \, h'\left(  \xi_{2j-1}(\tau_j^{\infty} ) \right) = -h( \xi_{2j-1}(\tau_j^{\infty} )) \in (-\sqrt{2},-1).
\]
Using \eqref{eq:first_bound},
\begin{equation}\label{eq:second_bound}
h'\left(  \xi_{2j-1}(\tau_j^{\infty}) \right) \in \left(  - \frac{\sqrt{2}}{\pi\, \left( j -\frac{1}{2}\right)},\ - \frac{1}{\pi\, j}\right).
\end{equation}
For $j$ large enough, this interval is contained in $(-2\epsilon,-\epsilon)$, and $[h']^{-1}(-2\epsilon,-\epsilon)$ contains only points of differentiability of the function $h'$. As a result, we may write
\[
h'\left(  \xi_{2j-1}(\tau_j^{\infty} ) \right) = \int_{\frac{\pi}{2}\, \left( j -\frac{1}{2}\right)}^{ \xi_{2j-1}(\tau_j^{\infty} )} h''(\xi)\, d\xi = - \int_{\frac{\pi}{2}\, \left( j -\frac{1}{2}\right)}^{ \xi_{2j-1}(\tau_j^{\infty} )} h(\xi)\, d\xi
\]
From \eqref{eq:second_bound} and the fact that $h(\xi)\geq 1$, it follows that:
\[
 - \frac{\sqrt{2}}{\pi\, \left( j -\frac{1}{2}\right)} \leq  - \int_{\frac{\pi}{2}\, \left( j -\frac{1}{2}\right)}^{ \xi_{2j-1}(\tau_j^{\infty} )} h(\xi)\, d\xi\leq - \left[  \xi_{2j-1}(\tau_j^{\infty} ) - \frac{\pi}{2}\, \left( j -\frac{1}{2}\right)\right]
\]
From \eqref{eq:second_bound} and the fact that $h(\xi)\leq \sqrt{2}$, it follows that: 
\[
- \frac{1}{\pi\, j}\geq - \int_{\frac{\pi}{2}\, \left( j -\frac{1}{2}\right)}^{ \xi_{2j-1}(\tau_j^{\infty} )} h(\xi)\, d\xi\geq 
-\sqrt{2}\, \left[  \xi_{2j-1}(\tau_j^{\infty} ) - \frac{\pi}{2}\, \left( j -\frac{1}{2}\right)\right].
\]

Combining these bounds, we find that 
\begin{equation}\label{eq:third_bound}
\frac{1}{\sqrt{2}\,\pi\, j} \leq \left[  \xi_{2j-1}(\tau_j^{\infty} ) - \frac{\pi}{2}\, \left( j -\frac{1}{2}\right)\right]\leq \frac{\sqrt{2}}{\pi\, \left( j -\frac{1}{2}\right)} .
\end{equation}

Letting $\xi_{2j-1}^{\infty} := \xi_{2j-1}(\tau_j^{\infty} )$, we have that 
\[
\tau_j^{\infty} -\tau_j = 
\frac{\xi_{2j-1}^{\infty}  h (\xi_{2j-1}^{\infty} )^2 -  \frac{\pi}{2}\, \left( j -\frac{1}{2}\right)  h( \frac{\pi}{2}\, \left( j -\frac{1}{2}\right))^2}{2\lambda^2} 
= \frac{\xi_j^{\infty} h (\xi_j^{\infty})^2 -  \pi  \left( j -\frac{1}{2}\right)}{2\lambda^2} 
\]
By \eqref{eq:third_bound} and the fact that $|h'|\leq 1$,
\[
|h (\xi_j^{\infty}) - \sqrt{2}| = |h (\xi_j^{\infty}) - h( \frac{\pi}{2}\, \left( j -\frac{1}{2}\right))|= \Big |\int_{ \frac{\pi}{2}\, \left( j -\frac{1}{2}\right)}^{\xi_j^{\infty}} h'(\zeta) \, d\zeta \Big | \leq |\xi_j^{\infty} -  \frac{\pi}{2}\, \left( j -\frac{1}{2}\right)| \leq  \frac{\sqrt{2}}{\pi\, \left( j -\frac{1}{2}\right)} .
\]
Using this bound together with \eqref{eq:third_bound} yields
\[
\tau_j^{\infty} -\tau_j   =   \frac{\xi_j^{\infty} - \frac{\pi}{2}  \left( j -\frac{1}{2}\right)}{\lambda^2} + \O\left( \frac{1}{j\, \lambda^2}\right) 
= \O\left( \frac{1}{j\, \lambda^2}\right) .
\]
\end{proof}

The fact that $\tau_j^{\infty}-\tau_j=o(1)$ as $j\rightarrow\infty$ suggests that $y_{2j-1}(\tau_j^{\infty})$ should get closer and closer to $y_{2j-1}(\tau_j) =\frac{\lambda}{\sqrt{2}}$ as $j\rightarrow\infty$. For this reason, it is convenient to rewrite our problem into coordinates that zoom into this point. This is our next result:

\begin{thm}\label{thm:mu} Fix $\lambda>0$. Let $j\geq 4$ and consider the coordinates 
\begin{equation}
\xi_{2j-1}(\tau) =2\tau y_{2j-1}(\tau)^2 , \qquad \xi_{2j}(\tau) =2\tau y_{2j}(\tau)^2, \qquad \tau \in [\tau_{j-2}, \tau_j].
\end{equation}
Then there exists some $j_0\in \N$ (independent of $\lambda$) such that for all $j\geq j_0$  we may write 
\begin{equation}\label{eq:xi_mu_formula}
\begin{split}
\xi_{2j-1}(\tau) & = \frac{\pi}{2}\, j - \frac{\pi}{4} - \mu_{2j-1}(\tau),\\
\xi_{2j}(\tau) & = \frac{\pi}{2}\, j - \frac{\pi}{4} + \mu_{2j}(\tau),
\end{split}
\end{equation}
where
\begin{equation}\label{eq:mu_bounds}
\lambda^2\, \frac{\tau_j - \tau}{5\, j}  \leq \mu_{2j-1}(\tau)\leq \frac{4}{\sqrt{j}},\qquad\quad
\frac{1}{\pi\, j}  \leq \mu_{2j}(\tau)\leq \frac{4}{\sqrt{j}} 
\end{equation}
for all $\tau \in [\tau_{j-2}, \tau_j]$.
\end{thm}
\begin{proof} 
{\bf Step 1. Implicit equation for $\mu$.}
The estimates on $\mu_{2j-1}$ and $\mu_{2j}$ follow from a fixed point argument, which we describe bellow. Recall that 
\[
\tau_{j-2} = \frac{\pi}{2\lambda^2}\, j - \frac{5\pi}{4\lambda^2} , \qquad \tau_{j} = \frac{\pi}{2\lambda^2}\, j - \frac{\pi}{4\lambda^2}
\]
We therefore introduce the following notation:
\[
\tau =  \frac{\pi}{2\lambda^2}\, j - \frac{3\pi}{4\lambda^2} + \frac{\zeta}{2\lambda^2} \qquad \mbox{with}\ \zeta\in \left[ - \pi , \pi\right].
\]
In this range of $\zeta$, the variable $\tau$ varies in the desired interval $[\tau_{j-2},\tau_j]$.

As detailed in \Cref{thm:curve_tau_xi}, $\xi_{2j-1}(\tau)$ (and $\xi_{2j}$) satisfies the equation:
\[
\xi_{2j-1}(\tau) \, h( \xi_{2j-1}(\tau))^2 = 2\tau \, \lambda^2.
\]
Using \eqref{eq:xi_mu_formula}, we rewrite this equation in terms of $\mu_{2j-1}$ and $\zeta$.
\begin{equation}\label{eq:implicit_mu}
\left( \frac{\pi}{2}\, j - \frac{\pi}{4} - \mu_{2j-1}(\zeta)\right) \, h \left( \frac{\pi}{2}\, j - \frac{\pi}{4} - \mu_{2j-1}(\zeta)\right)^2 =  \pi j - \frac{3\pi}{2} + \zeta \qquad \mbox{with}\ \zeta\in [- \pi , \pi].
\end{equation}
The existence of $\mu_{2j-1}$ is guaranteed by that of $\xi_{2j-1}$ which follows from \Cref{thm:bifurcations} in view of the range for $\tau$. Moreover, the function $h$ is $C^{\infty}$ in the interior of the interval where $\xi_{2j-1}$ lives. To carry out a more precise analysis, however, it is convenient to write $h \left( \frac{\pi}{2}\, j - \frac{\pi}{4} - \mu_{2j-1}(\zeta)\right)$ as a Taylor series. Indeed, note that 
\[
h \left( \frac{\pi}{2}\, j - \frac{\pi}{4} \right) =\sqrt{2}, \qquad  h' \left( \frac{\pi}{2}\, j - \frac{\pi}{4} \right)= 0.
\]
Note also that, in all $C^{\infty}$ points, we have that $h''=-h$. This implies that
\[
h^{(2n)} \left( \frac{\pi}{2}\, j - \frac{\pi}{4} \right) =(-1)^n\, \sqrt{2}, \qquad  h^{(2n+1)} \left( \frac{\pi}{2}\, j - \frac{\pi}{4} \right)= 0 \qquad \mbox{for all}\ n\geq 0.
\]
In particular, we have that 
\[
h \left( \frac{\pi}{2}\, j - \frac{\pi}{4} - \mu_{2j-1}(\zeta)\right) = \sqrt{2}\, \sum_{n=0}^{\infty} \frac{(-1)^n\, \mu_{2j-1}(\zeta)^{2n}}{(2n)!} ,
\]
as long as the series is well-defined, which we will soon show.

Therefore, we may formally rewrite \eqref{eq:implicit_mu} as
\begin{equation}\label{eq:implicit_mu_2}
\left( \pi\, j - \frac{\pi}{2} - 2\mu_{2j-1}(\zeta)\right) \, \left[\sum_{n=0}^{\infty} \frac{(-1)^n\, \mu_{2j-1}(\zeta)^{2n}}{(2n)!}\right]^2 =  \pi j - \frac{3\pi}{2} + \zeta \qquad \mbox{with}\ \zeta\in [- \pi , \pi].
\end{equation}

Gathering all terms up to cubic powers of $\mu_{2j-1}$ and writing the rest as a remainder leads to the formal identity:
\begin{equation}\label{eq:implicit_mu_3}
\pi - \zeta - 2 \mu_{2j-1}(\zeta) - \left( \pi j + \frac{\pi}{2} - 2\mu_{2j-1} (\zeta)\right) \, \mu_{2j-1}(\zeta)^2+ \Rc_{j,-} (\mu_{2j-1})(\zeta)=0,
\end{equation}
where, for any continuous function $v:[-\pi,\pi]\rightarrow \R$, we define
\begin{equation}\label{eq:R_mu}
\Rc_{j,\pm} (v)(\zeta) = \left( \pi\, j - \frac{\pi}{2} \pm 2v(\zeta)\right) \, \left[ \left(\sum_{n=1}^{\infty} \frac{(-1)^n\, v(\zeta)^{2n}}{(2n)!}\right)^2 + 2\, \sum_{n=2}^{\infty} \frac{(-1)^n\, v(\zeta)^{2n}}{(2n)!}\right].
\end{equation}

Finally, we may rewrite \eqref{eq:implicit_mu_3} as a contraction mapping problem $\mu_{2j-1}= \Phi_{-} (\mu_{2j-1})$ where
\begin{equation}\label{eq:Phi_mu}
\Phi_{\pm} (v)(\zeta) = \frac{\sqrt{[\pi\, j - \frac{\pi}{2} \pm 2v(\zeta)] \, [\pi - \zeta + \Rc_{j,\pm} (v)(\zeta) ] + 1} \pm 1}{\pi\, j - \frac{\pi}{2} \pm 2v(\zeta)}
\end{equation}
We note that a similar argument casts $\mu_{2j}$ as the solution to $\mu_{2j}= \Phi_{+} (\mu_{2j})$. We thus concentrate on $\Phi_{-}$ from now on.

The rest of the proof consists of showing that \eqref{eq:Phi_mu} is a well-defined contraction on the ball
\begin{equation}\label{eq:spaceX}
X= \left\{ v \in L^{\infty} ([-\pi,\pi]) \ \Big |\  \sup_{\zeta\in [-\pi,\pi]} |v(\zeta)| \leq \frac{4}{\sqrt{j}} \right\}
\end{equation}
for $j$ large enough. The local uniqueness of solutions to \eqref{eq:implicit_mu} will then guarantee that the solution found via the contraction mapping theorem necessarily satisfies \eqref{eq:xi_mu_formula}, i.e. $\mu_{2j-1}$ truly corresponds to $\xi_{2j-1}$.

{\bf Step 2. Bounds on remainder.} 
For $v\in X$, 
\begin{equation}\label{eq:remainder_upperb}
\begin{split}
|\Rc_{j,\pm} (v)(\zeta)| & \leq \pi\, j \, \left[ \left(\sum_{n=1}^{\infty} \frac{|v(\zeta)|^{2n}}{(2n)!}\right)^2 + 2\, \sum_{n=2}^{\infty} \frac{|v(\zeta)|^{2n}}{(2n)!}\right]\\
& \leq  \frac{64\pi}{j}\, \left[ \left(\sum_{n=1}^{\infty} \frac{8^{n-1}}{j^{n-1}(2n)!}\right)^2 + 2\, \sum_{n=2}^{\infty} \frac{8^{n-2}}{j^{n-2}(2n)!}\right]  \leq  \frac{16\pi}{j}\, (e^{16/j}  + e^{8/j})< \frac{64\pi}{j}
\end{split}
\end{equation}
for all $j\geq j_0$ large enough. Similarly, we have the following lower bound for $v\in X$
\begin{equation}\label{eq:remainder_lowerb}
\begin{split}
\Rc_{j,\pm} (v)(\zeta) & \geq j \, \left[\frac{v(\zeta)^4}{12}+  2\, \sum_{n=3}^{\infty} \frac{(-1)^n\, v(\zeta)^{2n}}{(2n)!}\right] \geq  j \, \left[\frac{v(\zeta)^4}{12} - \frac{16}{j}\, v(\zeta)^4 \, \sum_{n=3}^{\infty} \frac{ 8^{n-3}}{j^{n-3} (2n)!}\right]\\
& \geq  j \, v(\zeta)^4\, \left(\frac{1}{12} - \frac{2}{3 j}\, e^{8/j} \right) \geq \frac{j}{24} \, v(\zeta)^4\geq 0
\end{split}
\end{equation}
for all $j\geq j_0$ large enough. In particular, this lower bound shows that \eqref{eq:Phi_mu} is well-defined for all $v\in X$, as the terms inside the square root are non-negative. Moreover, the denominator is strictly positive for $v\in X$ provided $j\geq j_0$ large enough.

{\bf Step 3. Fixed point argument.}

Let us first show that $\Phi_{-} : X \rightarrow X$. Using \eqref{eq:Phi_mu}, \eqref{eq:remainder_upperb} and the fact that $v\in X$,
\[
\begin{split}
|\Phi_{-} (v)(\zeta)| & \leq \sqrt{\frac{\pi - \zeta + \Rc_{j,-} (v)(\zeta)}{\pi\, j - \frac{\pi}{2} - 2v(\zeta)} + \frac{1}{[\pi\, j - \frac{\pi}{2} - 2v(\zeta)]^2}}+ \frac{2}{\pi j}\\
& \leq \sqrt{\frac{2 + \frac{64}{j}}{j - \frac{1}{2} - \frac{8}{\pi \sqrt{j}}} + \frac{1}{[\pi\, j - \frac{\pi}{2} - \frac{8}{\sqrt{j}}]^2}}+ \frac{2}{\pi j} < \frac{4}{\sqrt{j}}
\end{split}
\]
for all $j\geq j_0$ sufficiently large.

In order to show the Lipschitz property, it suffices to bound the derivative of $\Phi_{-} (x)$ uniformly in $x\in [-4/\sqrt{j}, 4/\sqrt{j}]$. A straight-forward computation shows that 
\[
|\Phi_{-}'(x)| = \O( j^{-1} + |\Rc_{j,-}'(x)|) \qquad \mbox{for}\ x\in \left[-\frac{4}{\sqrt{j}}, \frac{4}{\sqrt{j}}\right].
\]
A computation in the spirit of \eqref{eq:remainder_upperb} shows that $|\Rc_{j,-}'(x)| =\O( j^{-1/2})$ uniformly in $x\in [-4/\sqrt{j}, 4/\sqrt{j}]$. In particular, this implies that $|\Phi_{-}'(x)| = \O(j^{-1/2})$. As a result, the Lipchistz constant of $\Phi_{-}$ can be made strictly smaller than 1 for $j$ sufficiently large.

By the contraction mapping theorem, there exists $\mu_{2j-1}\in X$ such that $\mu_{2j-1}= \Phi_{-} (\mu_{2j-1})$.
The upper bounds in \eqref{eq:mu_bounds} then follow from the fact that $\mu_{2j-1},\mu_{2j}$ belong to the space $X$.
In order to prove the lower bound, we use the upper bound \eqref{eq:mu_bounds} and the positivity of the remainder \eqref{eq:remainder_lowerb},
\[
\begin{split}
\mu_{2j-1}(\zeta) & = \frac{\sqrt{[\pi\, j - \frac{\pi}{2} - 2\mu_{2j-1}(\zeta)] \, [\pi - \zeta + \Rc_{j,-} (\mu_{2j-1})(\zeta) ] + 1} - 1}{\pi\, j - \frac{\pi}{2} - 2\mu_{2j-1}(\zeta)}\\
& \geq \sqrt{\frac{\pi - \zeta}{\pi\, j - \frac{\pi}{2} +\frac{8}{\sqrt{j}}} + \frac{1}{(\pi\, j - \frac{\pi}{2}+  \frac{8}{\sqrt{j}})^2}} - \frac{1}{\pi\, j - \frac{\pi}{2} + \frac{8}{\sqrt{j}}} \geq \frac{\pi - \zeta}{10j}
\end{split}
\]
for $j\geq j_0$ large enough. Using the fact that $\pi - \zeta = 2\lambda^2 ( \tau_j - \tau)$ yields \eqref{eq:mu_bounds}.

Similarly, we find that
\[
\mu_{2j}(\zeta)  = \frac{\sqrt{[\pi\, j - \frac{\pi}{2} + 2\mu_{2j}(\zeta)] \, [\pi - \zeta + \Rc_{j,+} (\mu_{2j})(\zeta) ] + 1} + 1}{\pi\, j - \frac{\pi}{2} + 2\mu_{2j}(\zeta)}\geq \frac{1}{\pi j},
\]
which concludes the proof of \eqref{eq:mu_bounds}.
\end{proof}

The previous theorem has various consequences. The first one is a lower bound for the distance between consecutive solutions to the implicit equation. 

\begin{cor}\label{thm:asymp_increment} Fix $\lambda>0$. There exists $j_0\in\N$ (independent of $\lambda$) such that for all $j\geq j_0$ and for $\tau\in [\tau_{j-2},\tau_j^{\infty}]$, we have
\begin{equation}
\left[ y_{2j}(\tau)-y_{2j-1}(\tau) \right] + \left[ y_{2j+2}(\tau)-y_{2j+1}(\tau) \right] \gtrsim \frac{\lambda}{j^2} \gtrsim \frac{1}{\lambda^3 \, \tau^2}
\end{equation}
where the implicit constant is independet of $\tau$, $j$ and $\lambda$.
\end{cor}
\begin{proof} By \Cref{thm:mu}, for all $j\geq j_0$ sufficiently large and for all $\tau\in [\tau_{j-2},\tau_j]$ we may write 
\[
\begin{split}
 y_{2j}(\tau)-y_{2j-1}(\tau) & = \frac{1}{\sqrt{2\tau}} \, \left( \sqrt{\xi_{2j}(\tau)} -  \sqrt{\xi_{2j-1}(\tau)}\right) \\
 & \geq  \frac{1}{\sqrt{2\tau}} \,  \left[ \xi_{2j}(\tau)- \xi_{2j-1}(\tau)\right]\, \inf_{\eta \in [\xi_{2j-1}(\tau), \xi_{2j}(\tau)]} \frac{1}{2\sqrt{\eta}}\\
 & \geq \frac{1}{\sqrt{2\pi \tau j}} \, \left[ \mu_{2j}(\tau) + \mu_{2j-1}(\tau)\right] \geq \frac{1}{2\pi^2\, \sqrt{\tau} j^{3/2}}.
 \end{split}
\]
In particular, we have that 
\begin{equation}\label{eq:increment}
 y_{2j}(\tau)-y_{2j-1}(\tau)\geq \frac{1}{2\pi^2\, j^{3/2}\, \sqrt{\tau}}\gtrsim \frac{\lambda}{j^2}  \qquad \mbox{for}\ \tau \in [\tau_{j-2},\tau_{j}].
\end{equation}
Using \eqref{eq:increment} evaluated at $j$ and $j+1$, together with the fact that $\tau_j^{\infty}\in [\tau_j,\tau_{j+1})$, yields the desired result.
\end{proof}

The next consequence is that the minimal solution $Y(\tau)$ admits a limit as $\tau\rightarrow\infty$, despite its infinite amount of jumps.

\begin{cor}\label{thm:Y_inf_limit}
Fix $\lambda>0$.  There exists some $j_0\in\N$ and some constant $C>0$ (both independent of $\lambda$), such that for all $j\geq j_0$ we have
\begin{equation}\label{eq:Y_inf_limit_1}
0 \leq  y_{2j+1}(\tau) - \frac{\lambda}{\sqrt{2}} \leq C\, \frac{\lambda}{j},
\end{equation}
uniformly in $\tau\in [\tau_{j-1},\tau_{j+1}]$. In particular, the minimal solution $Y(\tau)$ defined in \eqref{eq:big_Y} satisfies
\begin{equation}\label{eq:Y_inf_limit_2}
\lim_{\tau\rightarrow\infty} Y(\tau) =\frac{\lambda}{\sqrt{2}}.
\end{equation}
\end{cor}
\begin{proof} Let $\tau\in [\tau_{j-1},\tau_{j+1}]$. By \Cref{thm:mu}, we have that 
\[
y_{2j+1}(\tau)  = \sqrt{ \frac{\frac{\pi}{2}\, j  + \frac{\pi}{4} - \mu_{2j+1}(\tau)}{2\tau}} \leq \sqrt{ \frac{\frac{\pi}{2}\, j  + \frac{\pi}{4}}{2\tau_{j-1}}} = \sqrt{ \frac{\frac{\pi}{2}\, j  + \frac{\pi}{4}}{ 2\, \frac{\pi}{2\lambda^2} (j-\frac{3}{2})}} = \frac{\lambda}{\sqrt{2}}\, \sqrt{ 1 + \frac{4}{2j-3}}
\]
It follows that
\[
y_{2j+1}(\tau) - \frac{\lambda}{\sqrt{2}} \lesssim \frac{\lambda}{j},
\]
which concludes the proof of \eqref{eq:Y_inf_limit_1}.

In order to prove \eqref{eq:Y_inf_limit_2}, note that the range of $y_{2j+1}(\tau)$ in $\tau \in (\tau_{j}^{\infty}, \tau_{j+1}^{\infty}]$ is contained in 
\[
\left[ \frac{\lambda}{\sqrt{2}}, \max\{ y_{2j+1}(\tau_{j}^{\infty}),y_{2j+1}(\tau_{j+1}^{\infty})\} \right]
\]
 since $y_{2j+1}$ is decreasing until $\tau_{j+1}$ and increasing until it collides. Since $y_{2j+1}<y_{2j+3}$, we have that
\[
\left[ \frac{\lambda}{\sqrt{2}}, \max\{ y_{2j+1}(\tau_{j}^{\infty}),y_{2j+1}(\tau_{j+1}^{\infty})\} \right]\subset \left[ \frac{\lambda}{\sqrt{2}}, \max\{ y_{2j+1}(\tau_{j}^{\infty}),y_{2j+3}(\tau_{j+1}^{\infty})\} \right].
\]
Consequently, in order to show that $Y(\tau)\rightarrow \lambda/\sqrt{2}$ as $\tau\rightarrow\infty$ it suffices to show that 
\begin{equation}\label{eq:big_Y_limit}
\lim_{j\rightarrow\infty} y_{2j+1}(\tau_j^{\infty}) = \frac{\lambda}{\sqrt{2}}.
\end{equation}
Given that $y_{2j+1}$ is decreasing for all $\tau<\tau_{j+1}$, we have that $y_{2j+1}(\tau_j)\geq y_{2j+1}(\tau_j^{\infty})\geq \lambda/\sqrt{2}$. As a result, it suffices to show that 
\begin{equation}
\lim_{j\rightarrow\infty} y_{2j+1}(\tau_j) = \frac{\lambda}{\sqrt{2}}.
\end{equation}
This follows directly from \eqref{eq:Y_inf_limit_1}.
\end{proof}

This completes a thorough description of the solutions to the implicit equation \eqref{eq:def_h} including their long-time behavior. Before using these results to give bounds Gaussian integrals, we need a last result that will allow us to exploit these bounds under small variations of the parameter $\lambda$.

\subsection{Perturbation in $\lambda$}

Our next result explores the way in which solutions $y(\tau)$ depend on the parameter $\lambda>0$.

\begin{lem}\label{thm:perturb_lambda} Let $j\in\N$, $\lambda_0>0$ and consider a solution $y_{2j-1} (\tau;\lambda_0)$ (resp. $y_{2j}(\tau;\lambda_0)$) to \eqref{eq:def_F}. Fix $\tau \in (\widetilde{\tau}_{j-1}(\lambda_0),\tau_j^{\infty}(\lambda_0))$ (resp. $\tau \in (\widetilde{\tau}_{j}(\lambda_0),\tau_j^{\infty}(\lambda_0))$), then there exists a constant $d (\tau,\lambda_0)>0$ such that, for all 
\[
|\lambda-\lambda_0|\leq d (\tau,\lambda_0),
\]
the function $y_{2j-1} (\tau;\lambda)$ (resp. $y_{2j}(\tau;\lambda_0)$) is increasing (resp. decreasing) in $\lambda$ and
\begin{equation}\label{eq:perturb_lambda}
\begin{split}
 |y_{2j-1} (\tau;\lambda)-y_{2j-1}(\tau;\lambda_0)| & \lesssim_{\tau,\lambda_0} |\lambda-\lambda_0|,\\
  |y_{2j} (\tau;\lambda)-y_{2j}(\tau;\lambda_0)| & \lesssim_{\tau,\lambda_0} |\lambda-\lambda_0|.
 \end{split}
\end{equation}
\end{lem}
\begin{proof}  Consider the implicit equation \eqref{eq:def_F} with $\tau$ fixed and $\lambda$ as a variable:
\[
F(\tau; y,\lambda)= y\, h(2\tau y^2)-\lambda = 0,
\]
with $h$ defined in \eqref{eq:def_h}. Using the notation $\xi_{2j-1} (\lambda):=2\tau y_{2j-1}(\tau,\lambda)^2$, we note that
\[
\frac{\pa F }{\pa y} (\tau; y_{2j-1}(\tau,\lambda), \lambda) = h(\xi(\lambda)) + 2\xi(\lambda) h'(\xi (\lambda))
\]
The fact that $\tau \in (\widetilde{\tau}_{j-1}(\lambda_0),\tau_j^{\infty}(\lambda_0))$ implies that $\frac{\pa F }{\pa y}(\tau; y_{2j-1}(\tau,\lambda_0), \lambda_0) \neq 0$. The implicit function theorem then implies the existence of a small neighborhood of $\lambda_0$ with radius $d(\tau,\lambda_0)$ where $y$ can be solved as an implicit function of $\lambda$. By the local uniqueness, this solution $y(\tau,\lambda)$ must be exactly $y_{2j-1}(\tau,\lambda)$. Moreover, we have that $y_{2j-1}$ is locally differentiable in $\lambda$ and
\begin{equation}\label{eq:der_y_lambda}
\frac{\pa}{\pa\lambda} y_{2j-1} (\tau,\lambda) = \left( \frac{\pa F }{\pa y} (\tau, y_{2j-1} (\tau,\lambda), \lambda)\right)^{-1} = \frac{1}{h(\xi_{2j-1} (\lambda)) + 2\xi_{2j-1} (\lambda) h'(\xi_{2j-1} (\lambda))}.
\end{equation}
Note that the right-hand side of \eqref{eq:der_y_lambda} is positive as long as $y_{2j-1}$ exists (and in particular at $(\tau,\lambda_0)$) as explained in the proof of (v) in \Cref{thm:bifurcations}. This shows that $y_{2j-1}$ is locally increasing in $\lambda$. A similar argument shows that $y_{2j}$ is locally decreasing.

Finally, \eqref{eq:perturb_lambda} follows by making $d(\tau,\lambda_0)$ small enough so that 
\[
\frac{1}{2}\Big | \frac{\pa F }{\pa y} (\tau, y_{2j-1} (\tau,\lambda_0), \lambda_0) \Big |\leq  \Big | \frac{\pa F }{\pa y} (\tau, y_{2j-1} (\tau,\lambda), \lambda) \Big | \leq 2 \Big | \frac{\pa F }{\pa y} (\tau, y_{2j-1} (\tau,\lambda_0), \lambda_0)\Big | 
\]
for all $|\lambda-\lambda_0|\leq d(\tau,\lambda_0)$. The same argument works in the case of the solution $y_{2j}$.
\end{proof}

\begin{lem}\label{thm:perturb_lambda_2}  For $j\in\N$, the collision time $\tau_j^{\infty}(\lambda)$ is a continuous function of $\lambda>0$.
\end{lem}
\begin{proof}
By \Cref{thm:curve_tau_xi},  $\tau_j^{\infty}(\lambda)= \frac{\xi_j^{\infty} \, h(\xi_j^{\infty})^2}{2\,\lambda^2}$ where 
$\xi_j^{\infty}$ is the $j$-th positive solution to the equation:
\[
h(\xi) + 2\xi h'(\xi) =0.
\]
In particular $\xi_j^{\infty}$ is independent of $\lambda$ and thus $\tau_j^{\infty}(\lambda)$ is a continuous function of $\lambda$. In fact, recall that 
\[
\xi_j^{\infty} = \frac{\pi}{2} \left( j -\frac{1}{2}\right) + \O(\frac{1}{j}),
\]
as we proved in \eqref{eq:third_bound}.
\end{proof}

\section{Asymptotic expansion of Gaussian integrals}\label{sec:Laplace}

Our goal in this section is to obtain a sharp asymptotic development for Gaussian integrals. More precisely, based on \eqref{eq:sup}, we have that:
\begin{equation}\label{eq:Gaussian_integral_ball}
\P \left( \{\norm{u_{\varepsilon} (t)}_{L^{\infty}_x} > z_0 \, \varepsilon^{1-\delta}\}\cap \{ |\alpha|+|\beta| \leq 2c\varepsilon^{-\delta}\}\right) = \Ic (\varepsilon),
\end{equation}
where $u$ is the solution to \eqref{eq:beating} and
\begin{equation}\label{eq:Gaussian_integral}
\Ic (\varepsilon)= \int_{\tilde{\Ac}(t,\lambda_{\varepsilon})}  \frac{4a\,b}{\sigma_{\alpha}^2\, \sigma_{\beta}^2}  \exp \left( - \frac{a^2}{\sigma_{\alpha}^2}  - \frac{b^2}{\sigma_{\beta}^2} \right) \, da \, db,
\end{equation}
for
\begin{multline}
\tilde{\Ac}(t,\lambda) = \left\lbrace (a,b)\in \R_{+}^2 \ \Big | \ a+b\leq 2c\varepsilon^{-\delta},\  \sqrt{a^2 \cos^2 (2\varepsilon^2 t \,[a^2 + b^2]) + b^2  \sin^2 (2\varepsilon^2  t \,[a^2 +b^2]) } \right. \\
\left. + \sqrt{b^2 \cos^2 (2\varepsilon^2 t \,[a^2 + b^2]) + a^2 \sin^2 (2\varepsilon^2  t \,[a^2 + b^2]) } \geq \lambda \right\rbrace ,
\end{multline}
and $\lambda_{\varepsilon}= z_0 \, \varepsilon^{-\delta} - \O(\varepsilon^{\frac{1}{2} - \frac{3}{2}\delta})$ for fixed $z_0>0$, $\delta\in (0,1)$.

The constant $c=c(z_0)>0$ in \eqref{eq:Gaussian_integral_ball} is chosen large enough so that the complement is negligible, i.e.
\[
\P \left( \{\norm{u_{\varepsilon} (t)}_{L^{\infty}_x} > z_0 \, \varepsilon^{1-\delta}\}\cap \{ |\alpha|+|\beta| \geq 2c\varepsilon^{-\delta}\} \right) \leq \exp\left( -\frac{4c^2\varepsilon^{-2\delta}}{\sigma_{\alpha}^2 + \sigma_{\beta}^2}\right)
\]
by \eqref{eq:LDP_L1}, which can be made as small as desired by increasing $c>0$, as we previously did in \eqref{eq:fix_c}.

A simple rescaling  $(a,b) \mapsto \varepsilon^{-\delta} (a,b)$ yields
\begin{equation}
\Ic(\varepsilon)= \int_{\Ac (\tau,\lambda)}  \frac{4\varepsilon^{-4\delta}\, a\, b}{\sigma_{\alpha}^2\, \sigma_{\beta}^2}  \exp \left( - \varepsilon^{-2\delta}\, \frac{a^2}{\sigma_{\alpha}^2}  - \varepsilon^{-2\delta}\, \frac{b^2}{\sigma_{\beta}^2} \right) \, da \, db.
\end{equation}
with $\lambda= z_0 -  \O(\varepsilon^{\frac{1}{2} (1-\delta)})$, $\tau=\varepsilon^{2-2\delta} t$, and
\begin{multline}\label{eq:A_set}
\Ac(\tau,\lambda) = \left\lbrace (a,b)\in \R_{+}^2 \ \Big | \  a+b\leq 2c,\ \sqrt{a^2 \cos^2 (2\tau \,[a^2 + b^2]) + b^2  \sin^2 (2 \tau \,[a^2 +b^2]) } \right. \\
\left. + \sqrt{b^2 \cos^2 (2 \tau \,[a^2 + b^2]) + a^2 \sin^2 (2 \tau \,[a^2 + b^2]) } \geq \lambda \right\rbrace .
\end{multline}

The main difficulty in giving sharp upper and lower bounds on such integrals is the time-dependent domain, which is generally non-convex, and is thus not amenable to a classical Laplace-type approximation.

In the case of $\sigma_{\alpha}=\sigma_{\beta}$ such sharp asymptotic developments can be obtained by indirect means, as shown in \Cref{thm:main_equal_var}. When $\sigma_{\alpha}>\sigma_{\beta}$, one may exploit a linear approximation to give sharp upper and lower bounds for \eqref{eq:Gaussian_integral} provided $t\ll \varepsilon^{-2(1-\delta)}$, as shown in \Cref{thm:LDP_diff_linear}. 
Our next goal is to derive such an asymptotic expansion in the more complicated case where $t\gtrsim \varepsilon^{-2(1-\delta)}$.

In order to do so, we first show that the set $\Ac$ \eqref{eq:A_set} contains a set which entirely depends on the curve $F$ introduced in \eqref{eq:def_F}, namely
 \begin{equation}\label{eq:F_again}
 F(\lambda; \tau, y) = y h(2\tau y^2) - \lambda=0.
 \end{equation}
 
More precisely, we have the following inclusion:

\begin{claim}\label{thm:inclusion} Fix $\lambda>0$, $c(z_0)>0$ sufficiently large, and define
\[
\Bc (\tau,\tilde{\lambda}):= \{ a\in [0,c] \mid F(\tilde{\lambda} ; \tau, a)  \geq 0\} \times \{ b\in  [0,\varepsilon]\}.
\]
Then there exists a constant $C_1>0$ and some $\varepsilon_0(\lambda)$ such that for all $\varepsilon<\varepsilon_0$ the following holds. Let $\tilde{\lambda}= \lambda - C_1 \sqrt{\tau} \varepsilon - \varepsilon^{1/2}$, then
\[
\Bc (\tau,\tilde{\lambda}) \subset  \Ac (\tau,\lambda).
\]
\end{claim}
\begin{rk} Note that the choice to limit $b$ to $[0,\varepsilon]$ is somewhat arbitrary: one could choose any interval of the form $[0,\varepsilon^n]$, $n\in\N$. Such flexibility is important in order to potentially extend the range of $\tau$ to longer timescales. In our case, the range of $\tau$ is limited by the remainder terms in the Hamiltonian, cf.\ \Cref{thm:normal_form}. However, our argument generalizes to other Hamiltonians whose remainder terms are better behaved, see also \Cref{rk:longer_gamma}.
\end{rk}
\begin{proof} 
Suppose that $a\in (0,c)$ and $b\in (0,\varepsilon)$. Assuming $a^2 \cos^2 (2\tau \,a^2 )>\varepsilon$ we have:
\[
\begin{split}
\sqrt{a^2 \cos^2 (2\tau \,[a^2 +b^2]) +b^2  \sin^2 (2\tau \,[a^2 + b^2]) } & \geq \sqrt{a^2 \cos^2 (2\tau \,[a^2 + b^2])} \\
&\geq  \sqrt{a^2 [ \cos^2 (2\tau \,a^2 ) - \O( \varepsilon^2 \tau)]}\\
&\geq  \sqrt{a^2  \cos^2 (2\tau \,a^2 )}  - \O( \sqrt{\tau} \, \varepsilon )
\end{split}
\]
A similar argument for the other root yields
\begin{equation}\label{eq:B_first}
\Bc (\tau,\tilde{\lambda}) \cap \{a^2 \min\{ \sin^2 (2\tau \,a^2 ), \cos^2 (2\tau \,a^2 )\} >\varepsilon\}  \subset \Ac \left(\tau, \tilde{\lambda}- \O( \sqrt{\tau} \, \varepsilon )\right).
\end{equation}

Next suppose that $a^2 \cos^2 (2\tau \,a^2 )\leq \varepsilon$ (the other case is analogous). Then $ F(\tilde{\lambda}; \tau, a) \geq 0$ implies that 
\begin{equation}\label{eq:cos_small}
\sqrt{a^2 \sin^2 (2\tau \,a^2 )} \geq \tilde{\lambda} - \sqrt{\varepsilon}.
\end{equation}

Moreover, using \Cref{thm:existence}, $ F(\tilde{\lambda}; \tau, a) \geq 0$ implies $a\geq \tilde{\lambda}/\sqrt{2}$. In particular, for all $\varepsilon<\varepsilon_0 := \tilde{\lambda}^2/4$ we must have that 
\[
a^2 \cos^2 (2\tau \,a^2 ) + a^2 \sin^2 (2\tau \,a^2 )>2\varepsilon,
\]
implying that $a^2 \sin^2 (2\tau \,a^2 )>\varepsilon$. 

Using \eqref{eq:cos_small}, we have that
\[
\begin{split}
\sqrt{a^2 \cos^2 (2\tau \,[a^2 + b^2]) + b^2  \sin^2 (2 \tau \,[a^2 +b^2]) }  & + \sqrt{b^2 \cos^2 (2\tau \,[a^2 + b^2]) + a^2 \sin^2 (2 \tau \,[a^2 + b^2]) } \\
 & \geq  \sqrt{a^2 \sin^2 (2 \tau \,[a^2 + b^2])} \\
&\geq  \sqrt{a^2 \sin^2 (2 \tau a^2 )} - \O(\sqrt{\tau}  \varepsilon)\\
& \geq \tilde{\lambda} - \sqrt{\varepsilon}- \O(\sqrt{\tau}  \varepsilon).
\end{split} 
\]
In particular, we find that 
\begin{equation}\label{eq:B_second}
\Bc (\tau,\tilde{\lambda} ) \cap \{a^2 \min\{ \sin^2 (2\tau \,a^2 ), \cos^2 (2\tau \,a^2 )\} \leq \varepsilon\} \subset \Ac (\tau,  \tilde{\lambda}- \sqrt{\varepsilon}- \O(\varepsilon\, \sqrt{\tau})).
\end{equation}

Combining \eqref{eq:B_first} and \eqref{eq:B_second} with the fact that $\Ac (\tau, \lambda_1) \subset \Ac (\tau,\lambda_2)$ if $\lambda_1\geq \lambda_2$, we find that 
\[
\Bc (\tau,\tilde{\lambda}) \subset \Ac (\tau,  \lambda)
\] 
for $\lambda = \min\{ \tilde{\lambda}- \sqrt{\varepsilon}- \O( \varepsilon\sqrt{\tau}), \tilde{\lambda}+ \O( \varepsilon \sqrt{\tau})\}$. In particular, one may fix $C_1>0$ large enough and set
\[
\lambda = \tilde{\lambda} - C_1  \varepsilon \sqrt{\tau} - \sqrt{\varepsilon}.
\]
\end{proof}

Thanks to \Cref{thm:inclusion} and to the positivity of the integrand, we have the following lower bound:
\[
\Ic(\varepsilon)\geq  \int_{\Bc (\tau,\tilde{\lambda})}  \frac{4\varepsilon^{-4\delta}\, a\, b}{\sigma_{\alpha}^2\, \sigma_{\beta}^2}  \exp \left( - \varepsilon^{-2\delta}\, \frac{a^2}{\sigma_{\alpha}^2}  - \varepsilon^{-2\delta}\, \frac{b^2}{\sigma_{\beta}^2} \right) \, da \, db
\]
with $\tilde{\lambda}=z_0 - C_1 \sqrt{\tau} \varepsilon-C_2 \, \varepsilon^{\frac{1}{2} (1-\delta)}  - \varepsilon^{1/2}$. 

We are ready to exploit the results in \Cref{sec:implicit}.

\subsection{Resonant timescales}

In this subsection, we assume that $t= \tau \varepsilon^{-2(1-\delta)}$ for some $\tau\in (0,\infty)$.  Our goal is to prove the lower bound \eqref{eq:resonantLDP_diff_variance}.

\subsubsection{Point of continuity}

Let $Y(\lambda; \tau)$, as defined in \eqref{eq:big_Y}, be the minimal solution to the implicit equation \eqref{eq:def_F}. Let us assume that $\tau$ is a point of continuity of $Y(z_0,\tau)$. We then have the following result:

\begin{prop}\label{thm:LDP_resonant} Fix $z_0>0$, $\delta\in (0,1)$ and let $t= \tau \varepsilon^{-2(1-\delta)}$ for some $\tau\in (0,\infty)$. Suppose that $\tau$ is a point of continuity of $Y(z_0,\tau)$. Then there exists some $\varepsilon_0 (\tau, z_0)>0$ such that for all $\varepsilon\leq \varepsilon_0$ we have:
\[
\varepsilon^{2\delta} \log \P \left( \norm{u_{\varepsilon} (t)}_{L^{\infty}_x} > z_0 \, \varepsilon^{1-\delta}\right)\geq \frac{\Jc(z_0, \tau)^2}{\sigma_{\alpha}^2}
\]
with $\Jc$ defined in \eqref{eq:def_J}.
\end{prop}
\begin{proof}
By \eqref{eq:big_Y}, there exists some $j\in\N$ such that $Y(z_0;\tau)=y_{2j-1}(z_0;\tau)$.  Since $\tau$ is a point of continuity of $y_{2j-1}(z_0;\tau)$, we have that $\tau \in (\tau_{j-1}^{\infty}(z_0), \tau_j^{\infty}(z_0))$. In particular, 
 $\tau \in (\tau_{j-1}^{\infty}(\tilde{\lambda}), \tau_j^{\infty}(\tilde{\lambda}))$ for $|\tilde{\lambda}-z_0|\ll 1$ small enough (in terms of $z_0$), given the continuity of $\tau_j^{\infty}(\tilde{\lambda})$ in the variable $\tilde{\lambda}$, cf.\ \Cref{thm:perturb_lambda_2}.
 
Moreover, $\tau$ is a point of continuity of $y_{2j-1}(\tilde{\lambda};\tau)$ for 
\begin{equation}\label{eq:lambda_tilde}
\tilde{\lambda}=z_0 - C_1 \sqrt{\tau} \varepsilon - C_2 \varepsilon^{\frac{1}{2} (1-\delta)}-  \varepsilon^{1/2}
\end{equation}
as long as $\varepsilon< \varepsilon_0(\tau,z_0)$ is small enough. If $\tau\geq \tau_1(\tilde{\lambda})$, we consider $y_{2j}(\tilde{\lambda}; \tau)$, the second smallest solution to \eqref{eq:def_F}. If $\tau<\tau_1$, we set by convention $y_{2j}(\tilde{\lambda}; \tau)\equiv \tilde{\lambda}$ instead. We then have that the following inclusion:
\begin{equation}\label{eq:inclusion}
[y_{2j-1}(\tilde{\lambda}; \tau), y_{2j}(\tilde{\lambda}; \tau)] \times [0,\varepsilon]_b \subset \Bc (\tau,\tilde{\lambda}).
\end{equation}

Setting $a_0=y_{2j-1}(\tilde{\lambda}; \tau)$, we have that 
\begin{equation}\label{eq:Gaussian_integral_toporder}
\mathcal{I}(\varepsilon) \geq \exp \left( - \varepsilon^{-2\delta}\, \frac{a_0^2}{\sigma_{\alpha}^2}\right)\, \mathcal{I}_{\mathrm{error}} (\varepsilon)
\end{equation}
where 
\begin{equation}\label{eq:Gaussian_integral_error}
\begin{split}
 \mathcal{I}_{\mathrm{error}} (\varepsilon) =&\  \int_{y_{2j-1}(\tilde{\lambda}; \tau)}^{y_{2j}(\tilde{\lambda}; \tau)} \int_{0}^{\varepsilon}  \frac{4\varepsilon^{-4\delta}\, a\,b}{\sigma_{\alpha}^2\, \sigma_{\beta}^2}  \exp \left( - \varepsilon^{-2\delta}\, \frac{a^2- a_0^2}{\sigma_{\alpha}^2}  - \varepsilon^{-2\delta}\, \frac{b^2}{\sigma_{\beta}^2} \right) \, db \, da
\end{split}
\end{equation}

Direct integration yields:
\begin{equation}\label{eq:bound_Ierror}
\begin{split}
 \mathcal{I}_{\mathrm{error}} (\varepsilon) & =  \int_{y_{2j-1}(\tilde{\lambda}; \tau)}^{y_{2j}(\tilde{\lambda}; \tau)} \frac{2\varepsilon^{-2\delta}\, a}{\sigma_{\alpha}^2}  \exp \left( - \varepsilon^{-2\delta}\, \frac{a^2- a_0^2}{\sigma_{\alpha}^2} \right) \, \left( 1- \exp (-\varepsilon^{2-2\delta}/\sigma_{\beta}^2) \right)\, da\\
 & \gtrsim \varepsilon^{2(1-\delta)}\,  \int_{y_{2j-1}(\tilde{\lambda}; \tau)}^{ y_{2j}(\tilde{\lambda}; \tau)} \frac{2\varepsilon^{-2\delta}\, a}{\sigma_{\alpha}^2}  \exp \left( - \varepsilon^{-2\delta}\, \frac{a^2- a_0^2}{\sigma_{\alpha}^2} \right) \, da\\
 & \gtrsim \varepsilon^{2(1-\delta)}\, \left[ 1- \exp \left( - \varepsilon^{-2\delta}\, \frac{ y_{2j}(\tilde{\lambda}; \tau)^2- y_{2j-1}(\tilde{\lambda}; \tau)^2}{\sigma_{\alpha}^2} \right) \right]
 \end{split}
\end{equation}

By \Cref{thm:perturb_lambda} and \eqref{eq:lambda_tilde}, we have that 
\[
\begin{split}
y_{2j}(\tilde{\lambda}; \tau)- y_{2j-1}(\tilde{\lambda}; \tau) & = y_{2j}(z_0; \tau)- y_{2j-1}(z_0; \tau) + \O_{z_0, \tau} ( |\tilde{\lambda}-z_0|) \\
& = y_{2j}(z_0; \tau)- y_{2j-1}(z_0; \tau) + \O_{z_0, \tau} ( \varepsilon^{(1-\delta)/2} )
\end{split}
\]
as long as $\varepsilon\leq \varepsilon_0$ is small enough so that $\varepsilon_0^{(1-\delta)/2}< d(\tau,z_0)$ in \Cref{thm:perturb_lambda}. Given that $\tau$ is a point of continuity of $Y(z_0,\tau)=y_{2j-1}(z_0; \tau)$, we have that 
the quantity 
\begin{equation}\label{eq:Delta_sol}
\Delta(z_0;\tau):= y_{2j}(z_0; \tau)- y_{2j-1}(z_0; \tau) >0.
\end{equation}
As a result,
\[
y_{2j}(\tilde{\lambda}; \tau)^2- y_{2j-1}(\tilde{\lambda}; \tau)^2 =[y_{2j}(\tilde{\lambda}; \tau)- y_{2j-1}(\tilde{\lambda}; \tau)]\cdot [y_{2j}(\tilde{\lambda}; \tau)+ y_{2j-1}(\tilde{\lambda}; \tau)]  \geq \frac{\Delta(z_0;\tau)}{2} \, z_0.
\]
This bound, together with \eqref{eq:bound_Ierror}, yields
\[
 \mathcal{I}_{\mathrm{error}} (\varepsilon)\gtrsim_{z_0,\tau}  \varepsilon^{2(1-\delta)}.
 \]
As a result,  \eqref{eq:Gaussian_integral_toporder} implies the bound
\[
\varepsilon^{2\delta} \log \Ic (\varepsilon) \geq -\frac{y_{2j-1}(\tilde{\lambda}; \tau)^2}{\sigma_{\alpha}^2} - o(1).
\]

Using \Cref{thm:perturb_lambda} and \eqref{eq:lambda_tilde}, we obtain the following equality for all $\varepsilon\leq \varepsilon_0 (\tau,z_0)$
\[
y_{2j-1}(\tilde{\lambda}; \tau) = y_{2j-1}(z_0; \tau) - \O_{z_0,\tau} \left( \varepsilon^{(1-\delta)/2}\right) .
\]
The fact that $\Jc (z_0, \tau)=y_{2j-1}(z_0; \tau)$ at all points of continuity finishes the proof.
\end{proof}

\subsubsection{Point of discontinuity}

Consider the minimal solution $Y(\lambda; \tau)$, as defined in \eqref{eq:big_Y}, to the implicit equation \eqref{eq:def_F}. Let us assume that $\tau$ is a jump point of $Y(z_0,\tau)$, i.e. $\tau = \tau_j^{\infty}$ for some $j\in\N$. In that case, we have the following:

\begin{prop}\label{thm:LDP_resonant2} Fix $z_0>0$, $\delta\in (0,1)$ and let $t= \tau \varepsilon^{-2(1-\delta)}$ for some $\tau\in (0,\infty)$. Suppose that $\tau$ is a jump point of $Y(z_0,\tau)$. Then there exists some $\varepsilon_0 (\tau, z_0)>0$ such that for all $\varepsilon\leq \varepsilon_0$ we have:
\[
\varepsilon^{2\delta} \log \P \left( \norm{u_{\varepsilon} (t)}_{L^{\infty}_x} > z_0 \, \varepsilon^{1-\delta}\right)\geq \frac{\Jc(z_0, \tau)^2}{\sigma_{\alpha}^2}
\]
with $\Jc$ defined in \eqref{eq:def_J}.
\end{prop}
\begin{proof}
One may repeat the proof of \Cref{thm:LDP_resonant} but, instead of using \eqref{eq:inclusion}, we exploit the next two smallest solutions, i.e.
\[
[y_{2j+1}(\tilde{\lambda}; \tau), y_{2j+2}(\tilde{\lambda}; \tau)] \times [0,\varepsilon]_b \subset \Bc (\tau,\tilde{\lambda}).
\]
The same strategy of proof then leads to
\[
\varepsilon^{2\delta} \log \Ic (\varepsilon) \geq -\frac{y_{2j+1}(z_0; \tau)^2}{\sigma_{\alpha}^2} - o(1),
\]
and since $\tau$ is a jump point, \eqref{eq:def_J} implies that $\Jc (z_0, \tau)=y_{2j+1}(z_0; \tau)$, which concludes the proof.
\end{proof}

\subsection{Super-resonant timescales}\label{sec:super_res_LDP}

As we noticed in the case of the resonant times, a key tool to approximate our Gaussian integral is showing that the set $\{ F(\tau,y;z_0)\geq 0\}$ contains an interval of the form $[a_0,a_1)$, where $a_0,a_1$ are the smallest solutions to $F(\tau,y;z_0)=0$. In particular, we fundamentally use that the distance $a_1-a_0$ is strictly positive, cf.\ \eqref{eq:Delta_sol}.

As $\tau\rightarrow\infty$, however, the distance between consecutive solutions tends to zero. Consequently, the arguments above are insufficient since all intervals formed by two consecutive solutions $[a_0,a_1)$ collapse in the limit $\tau\rightarrow\infty$. In order to overcome this difficulty, we exploit a lower bound on the measure of the union of two such intervals. 

\begin{proof}[Proof of \Cref{thm:super_resonantLDP_diff_variance}]
{\bf Upper bound.} 
Let $\varepsilon$ be sufficiently small for \Cref{thm:approximation} to hold. Define the events
\[
\Ac (\lambda) = \{  \norm{u(t)}_{L^{\infty}_x (\T)} \geq  \lambda\},\quad \Bc_1 ( \lambda ) =\{ |\alpha| +|\beta| \geq  \lambda \},\quad  \Bc_2 (\lambda ) = \{ \sqrt{|\alpha|^2 + |\beta|^2} \geq  \lambda\}.
\]

By \eqref{eq:LDP_L1}, we have that 
\begin{equation}\label{eq:fix_c2}
\varepsilon^{2\delta}\, \log \P ( \Bc_1 (2c\varepsilon^{-\delta}) ) \leq -\frac{4c^2}{\sigma_{\alpha}^2+\sigma_{\beta}^2} < - 100 \, \frac{z_0^2}{\sigma_{\alpha}^2+\sigma_{\beta}^2}
\end{equation}
by fixing $c> 5 z_0$.

Note that 
\[
\P(\Ac ( z_0 \varepsilon^{1-\delta})) = \P(\Ac ( z_0 \varepsilon^{1-\delta}) \cap  \Bc_1 (2c\varepsilon^{-\delta})^c ) +  \P(\Ac ( z_0 \varepsilon^{1-\delta}) \cap  \Bc_1 (2c\varepsilon^{-\delta}) ) .
\]
By \eqref{eq:fix_c2}, the second summand is negligible since
\[
0 \leq \P(\Ac ( z_0 \varepsilon^{1-\delta}) \cap  \Bc_1 (2c\varepsilon^{-\delta}) ) \leq \P ( \Bc_1 (2c\varepsilon^{-\delta}) ) \leq \exp\left( - 100 \, \frac{z_0^2\varepsilon^{-2\delta}}{\sigma_{\alpha}^2+\sigma_{\beta}^2}\right).
\]
We highlight that our equation \eqref{eq:beating} is a.s.\ globally well-posed and thus the solution is well-defined in the set $\Bc_1 (2c\varepsilon^{-\delta})$. Finally, \eqref{eq:general_upper_bound} and \eqref{eq:LDP_L2} yield
\begin{equation}\label{eq:step_L2_2}
 \log \P(\Ac ( z_0 \varepsilon^{1-\delta}) \cap  \Bc_1 (2c\varepsilon^{-\delta})^c ) \leq \log \P ( \Bc_2 ( \frac{z_0 \varepsilon^{-\delta}}{\sqrt{2}} - \O ( \varepsilon^{\frac{1}{2}-\frac{3}{2}\delta}))) \leq -\frac{z_0^2}{2\sigma_{\alpha}^2}\, \varepsilon^{-2\delta} + o( \varepsilon^{-2\delta}).
\end{equation}

{\bf Lower bound.} Suppose that $t=t(\varepsilon)>0$ satisfies
\begin{equation}\label{eq:t_eps}
\lim_{\varepsilon\rightarrow 0^{+}} t(\varepsilon)\varepsilon^{2(1-\delta)}=\infty, \qquad t(\varepsilon)\lesssim \varepsilon^{-\frac{5}{2}(1-\delta)+\kappa} \ \mbox{for}\ \kappa>0.
\end{equation}
Let $(\varepsilon_n)_{n\in\N}$ be any sequence tending monotonically to zero, let $t_n=t(\varepsilon_n)$ and let $\tau_n =\varepsilon_n^{2(1-\delta)} t_n$. By \eqref{eq:t_eps}, $(\tau_n)_{n\in\N}$ is a sequence tending to $+\infty$. Then 
\[
\P \left( \norm{u_{\varepsilon_n} (t_n)}_{L^{\infty}_x} > z_0 \, \varepsilon_n^{-\delta}\right)= \int_{\mathcal{A}(\tau_n,\lambda_n)}  \frac{4\varepsilon_n^{-4\delta}\, a\, b}{\sigma_{\alpha}^2\, \sigma_{\beta}^2}  \exp \left( - \varepsilon_n^{-2\delta}\, \frac{a^2}{\sigma_{\alpha}^2}  - \varepsilon_n^{-2\delta}\, \frac{b^2}{\sigma_{\beta}^2} \right) \, da \, db:=\Ic (\varepsilon_n),
\]
with $\lambda_n= z_0 -  \O(\varepsilon_n^{(1-\delta)/2})$ and $\Ac$ defined in \eqref{eq:A_set}. By \Cref{thm:inclusion}, we have that
\[
\Ic (\varepsilon_n)\geq \int_{\Bc(\tau_n,\tilde{\lambda}_n)}  \frac{4\varepsilon_n^{-4\delta}\, a\, b}{\sigma_{\alpha}^2\, \sigma_{\beta}^2}  \exp \left( - \varepsilon_n^{-2\delta}\, \frac{a^2}{\sigma_{\alpha}^2}  - \varepsilon_n^{-2\delta}\, \frac{b^2}{\sigma_{\beta}^2} \right) \, da \, db
\]
where 
\[
\Bc (\tau,\lambda):= \{ a\in [0,c ] \mid F(\lambda; \tau, a)  \geq 0\} \times \{ b\in  [0,\varepsilon_n]\},
\]
and $\tilde{\lambda}_n=z_0  - C_1 \sqrt{\tau_n} \varepsilon_n -  C_2 \varepsilon_n^{(1-\delta)/2} -  \varepsilon_n^{1/2}$. Note that $\sqrt{\tau_n} \varepsilon_n = O(\varepsilon_n^{3/4})$ as $n\rightarrow\infty$ in view of the rightmost condition in \eqref{eq:t_eps}.

For $n$ sufficiently large, let $y_{2j-1}(\tau_n, \tilde{\lambda}_n), \ldots ,y_{2j+2}(\tau_n, \tilde{\lambda}_n)$ be the four smallest solutions (in order) to $F(\tilde{\lambda}_n; \tau_n, a)=0$, where $j=j(n)\rightarrow \infty$ as $n\rightarrow\infty$.
In particular, this implies that $\tau_n\in ( \tau_{j-1}^{\infty} (\tilde{\lambda}_n),  \tau_{j}^{\infty} (\tilde{\lambda}_n)]$. Note that
\[
[ y_{2j-1}(\tau_n, \tilde{\lambda}_n),y_{2j}(\tau_n, \tilde{\lambda}_n)] \cup [y_{2j+1}(\tau_n, \tilde{\lambda}_n),y_{2j+2}(\tau_n, \tilde{\lambda}_n)] \subset \{ a\in [0,c] \mid F(\tilde{\lambda}_n; \tau_n, a)  \geq 0\}.
\]
By \Cref{thm:asymp_increment}, for $n$ sufficiently large, at least one of these intervals has length bounded below by a multiple of $\tilde{\lambda}_n^{-3} \tau_n^{-2}$. Let this interval be $[y_{2j+1}(\tau_n, \tilde{\lambda}_n),y_{2j+2}(\tau_n, \tilde{\lambda}_n)]$, the argument being analogous in the other case.

As in the proof of \Cref{thm:LDP_resonant}, we let $a_0=y_{2j+1}(\tau_n, \tilde{\lambda}_n)$, and we set
\begin{equation}\label{eq:Gaussian_integral_toporder_2}
\Ic (\varepsilon_n) \geq \exp \left( - \varepsilon_n^{-2\delta}\, \frac{a_0^2}{\sigma_{\alpha}^2}\right)\, \Ic_{\mathrm{error}} (\varepsilon_n)
\end{equation}
where 
\begin{equation}\label{eq:Gaussian_integral_error_2}
\begin{split}
 \Ic_{\mathrm{error}} (\varepsilon_n) = &\  \int_{y_{2j+1}(\tau_n, \tilde{\lambda}_n)}^{y_{2j+2}(\tau_n, \tilde{\lambda}_n)} \int_{0}^{\varepsilon_n}  \frac{4\varepsilon_n^{-4\delta}\, a\,b}{\sigma_{\alpha}^2\, \sigma_{\beta}^2}  \exp \left( - \varepsilon_n^{-2\delta}\, \frac{a^2- a_0^2}{\sigma_{\alpha}^2}  - \varepsilon_n^{-2\delta}\, \frac{b^2}{\sigma_{\beta}^2} \right) \, db \, da
\end{split}
\end{equation}
Direct integration yields:
\[
\begin{split}
 \Ic_{\mathrm{error}} (\varepsilon_n) & =   \int_{y_{2j+1}(\tau_n, \tilde{\lambda}_n)}^{y_{2j+2}(\tau_n, \tilde{\lambda}_n)} \frac{2\varepsilon_n^{-2\delta}\, a}{\sigma_{\alpha}^2}  \exp \left( - \varepsilon_n^{-2\delta}\, \frac{a^2- a_0^2}{\sigma_{\alpha}^2} \right) \, \left( 1- \exp (-\varepsilon_n^{2-2\delta}/\sigma_{\beta}^2) \right)\, da\\
 & \gtrsim \varepsilon_n^{2(1-\delta)}\,    \int_{y_{2j+1}(\tau_n, \tilde{\lambda}_n)}^{y_{2j+2}(\tau_n, \tilde{\lambda}_n)} \frac{2\varepsilon_n^{-2\delta}\, a}{\sigma_{\alpha}^2}  \exp \left( - \varepsilon_n^{-2\delta}\, \frac{a^2- a_0^2}{\sigma_{\alpha}^2} \right) \, da\\
 & \gtrsim \varepsilon_n^{2(1-\delta)}\, \left[ 1- \exp \left( - \varepsilon_n^{-2\delta}\, \frac{ y_{2j+2}(\tau_n, \tilde{\lambda}_n)^2- y_{2j+1}(\tau_n, \tilde{\lambda}_n)^2}{\sigma_{\alpha}^2} \right) \right]
 \end{split}
\]
By \Cref{thm:asymp_increment} and \eqref{eq:t_eps},
\[
\begin{split}
y_{2j+2}(\tau_n, \tilde{\lambda}_n)^2- y_{2j+1}(\tau_n, \tilde{\lambda}_n)^2 & = [y_{2j+2}(\tau_n, \tilde{\lambda}_n)- y_{2j+1}(\tau_n, \tilde{\lambda}_n)]\, [y_{2j+2}(\tau_n, \tilde{\lambda}_n)+ y_{2j+1}(\tau_n, \tilde{\lambda}_n)]\\
& \geq \sqrt{2}\tilde{\lambda}_n\cdot \tilde{\lambda}_n^{-3} \tau_n^{-2} \gtrsim \frac{1}{z_0^2}\,\varepsilon_n^{1-\delta}.
\end{split}
\]
In particular,
\[
 \Ic_{\mathrm{error}} (\varepsilon_n)  \gtrsim \varepsilon_n^{3(1-\delta)} .
\]

By \eqref{eq:Gaussian_integral_toporder_2},
\[
\Ic (\varepsilon_n) \gtrsim \varepsilon_n^{3(1-\delta)} \, \exp \left( - \varepsilon_n^{-2\delta}\, \frac{y_{2j+1}(\tau_n, \tilde{\lambda}_n)^2}{\sigma_{\alpha}^2}\right).
\]
By \Cref{thm:Y_inf_limit}, for $n$ sufficiently large we have
\[
\begin{split}
\varepsilon_n^{2\delta} \log \P \left( \norm{u_{\varepsilon_n} (t_n)}_{L^{\infty}_x} > z_0 \, \varepsilon_n^{-\delta}\right) & \geq - \frac{y_{2j+1}(\tau_n, \tilde{\lambda}_n)^2}{\sigma_{\alpha}^2} + o(1)\\
& \geq -\frac{\tilde{\lambda}_n^2}{2\sigma_{\alpha}^2}  + o(1) \geq -\frac{z_0^2}{2\sigma_{\alpha}^2}  + o(1),
\end{split}
\]
which concludes the proof of the lower bound.
\end{proof}

\bibliographystyle{hsiam}
\bibliography{references_two_modes}
\end{document}